\title{{\bf Hypersurface Singularities in any Characteristic}\\
{\large In Memoriam Wolfgang Ebeling}}
\author{ \bf{Gert-Martin Greuel}}

\date{}

\documentclass[11pt]{article}
\usepackage{amsmath, amsthm, amssymb, mathrsfs, amscd, amsthm, wasysym, amscd,amsfonts,graphicx, mathtools}
\usepackage{indentfirst,url,xypic, xcolor}
\usepackage{lipsum} 
\usepackage[all]{xy}
\usepackage{url}
\usepackage[colorlinks,plainpages]{hyperref}
\usepackage{verbatim}
\usepackage{fancyvrb}
\usepackage{listings}

\hypersetup{
	colorlinks=true,
	linkcolor=blue,
	filecolor=magenta,      
	urlcolor=cyan,
	citecolor=blue
}

\DeclareMathOperator{\rank}{rank}

\DeclareMathOperator{\id}{id}

\DeclareMathOperator{\supp}{Supp}

\DeclareMathOperator{\modular}{mod}
\DeclareMathOperator{\ord}{ord}

\DeclareMathOperator{\Aut}{Aut}
\DeclareMathOperator{\jet}{jet}

\DeclareMathOperator{\rsim}{\stackrel{r}{\sim}}
\DeclareMathOperator{\csim}{\stackrel{c}{\sim}}

\DeclareMathOperator{\Spec}{Spec}
\DeclareMathOperator{\Max}{Max}

\DeclareMathOperator{\Char}{char}

\def\x{\langle x \rangle}

\newtheorem{Definition}{Definition}[section]
\newtheorem{Theorem}[Definition]{Theorem}
\newtheorem{Remark}[Definition]{Remark}
\newtheorem{Proposition}[Definition]{Proposition }
\newtheorem{Corollary}[Definition]{Corollary}
\newtheorem{Example}[Definition]{Example}
\newtheorem{Lemma}[Definition]{Lemma}
\newtheorem{conjecture}[Definition]{Conjecture}

\newcommand{\N}{{\mathbb N}}

\newcommand{\Z}{{\mathbb Z}}
\newcommand{\R}{{\mathbb R}}
\newcommand{\C}{{\mathbb C}}
\newcommand{\Q}{{\mathbb Q}}
\renewcommand{\P}{{\mathbb P}}

\newcommand{\kk}{{\mathcal K}}

\newcommand{\kq}{{\mathcal Q}}
\newcommand{\kr}{{\mathcal R}}

\newcommand{\fm}{\mathfrak{m}}

\newcommand{\fp}{\mathfrak{p}}
\newcommand{\fq}{\mathfrak{q}}

\newcommand{\bx}{\boldsymbol{x}}

\newcommand{\beps}{\boldsymbol{\varepsilon}}
\newcommand{\eps}{\varepsilon}

\DeclareMathOperator{\mt}{mt}

\def\<bx{\langle \bx \rangle}

\begin{document}
\maketitle

\section*{Abstract}
 {We give an overview of the fundamental definitions and results concerning hypersurface singularities, defined by convergent power series over an arbitrary real valued field. This approach combines, on the one hand, the classical case of analytic power series over the complex numbers with formal power series over arbitrary fields, but on the other hand, it goes significantly beyond that.  
Besides  general definitions and  basic results, we report on the classification of contact simple and right simple singularities in positive characteristic. 
Some of the results are new in this general setting, for which we provide complete proofs.}

\tableofcontents

\section*{Introduction}
While the theory of singularities of complex analytic spaces and germs  is covered in several textbooks, the study of singularities of algebraic varieties and formal power series defined over a field of arbitrary characteristic has not yet been fully explored. 
This article provides an overview of the fundamental definitions and results concerning hypersurface singularities of any characteristic.

Many of the topics discussed here have counterparts over the complex numbers. But in positive characteristics, the results often differ significantly and require a different proof in each case. This overview is based on Chapter 3.1 of the second edition of the textbook "Introduction to Singularities and Deformations" (see  \cite{GLS25}) as well as on the two articles \cite{GP25} and \cite{GP25a}.

While \cite{GLS25} considers formal power series, here we deal more generally with convergent power series over a real-valued field $K$ of arbitrary characteristic. The main new ingredient we use is Grauert's approximation theorem and the existence of a semiuniversal deformation for isolated singularities $f \in K\{x\}$, the ring of convergent power series over $K$ (from \cite{GP25a}).
One advantage is that we do not need to distinguish between formal power series over arbitrary fields and the classical cases of convergent power series over $\R$ and $\C$. But this approach covers also many other interesting fields (see Remark \ref{rm.rvf}). Moreover, several results of this article are new for $K\{x\}$.\\

Here is a short overview. In Section \ref{sec.invariants} we generalize the classical definitions and results of complex analytic hypersurface singularities to power series $f\in K\{x\}$. This includes Milnor and Tjurina number, finite determinacy bounds (using Grauert's approximation theorem) and, finally, a generalization of the Mather-Yau theorem. Several results are new in this setting for which we give complete proofs.

Section \ref{sec.split} treats versal unfoldings and deformations and the important splitting lemma in any characteristic for $K\{x\}$, with full proofs.

In Section \ref{sec.class1} resp. Section \ref{sec.class2}  we report on the classification of contact resp. right simple singularities in positive characteristic. While $K$ may be an arbitrary real valued field in the previous sections, we assume for the classification that $K$ is algebraically closed.\\

We first fix some notations.  $K$ will always denote a field of characteristic $p \ge 0$ with a real valuation $| \, \ | :K \to \R_{\ge 0}$\,\footnote{This means: $ |a| = 0 \Leftrightarrow a=0$, $ |a||b| = |ab|$, and $|a+b| \le |a|+|b|$. $K$ is also called a normed field.}, except said otherwise. 
$K[[x]]$, $x=(x_1,\ldots,x_n)$, denotes the formal power series ring.
For each \mbox{${\eps}\in
  (\R_{>0})^n$}, we
define a map 
$\|\phantom{f}\|_{\eps}:K[[x]]\to \R_{>0}\cup  \{\infty\}$
   by setting for a formal power series
   \mbox{$f  =  \sum_{{\alpha}\in\N^n}c_{\alpha} {x}^{\alpha} $}, $c_\alpha \in K$, 
   $\alpha=(\alpha_1,...,\alpha_n)$, 
   $x^\alpha=x_1^{\alpha_1}\cdot...\cdot x_n^{\alpha_n}$,
$$\|f\|_{\eps}:= \sum_{{\alpha} \in \N^n} |c_{\alpha}|\cdot
{\eps}^{{\alpha}}\quad \in \R_{\ge 0}\cup \{\infty\}\,.$$
The power series
\mbox{$f  =  \sum_{{\alpha}\in\N^n}c_{\alpha} {x}^{\alpha} $}
is called {\em convergent}
iff there exists a real vector \mbox{$\eps\in(\R_{>0})^n$}
such that \mbox{$ \| f \|_{\eps} < \infty $}.
We denote by 
$$K\{x\}=K\{x_1,...,x_n\} \subseteq K[[x]]$$
the {\em convergent (or analytic) power series ring} over $K$. The analytic $K$-algebra 
$K\{ x\}$ is a Noetherian, integral and factorial local ring  with maximal ideal $\fm =  \langle x_1,...,x_n \rangle$, and morphisms being morphisms of local $K$-algebras (see \cite[Section I.1.2]{GLS25}). 
We denote  the {\em order} of $f \in K\{x\}$ 
by 
$$\ord(f) := \max \{l \mid f \in \fm^l\}\text{ if }f \ne 0\text{ and }\ord(0)=\infty.$$ 

 If  $K$  is complete w.r.t. $| \,\ |$ (every Cauchy sequence with respect to $| \,\ |$ converges in $K$) then $K$ is called a {\em complete real valued field}.
Note that $\bar K\{x\}\cap K\{x\}=K\{x\}$, where $\bar K$ denotes the completion of $K$. See Remark \ref{rm.rvf} for examples of real valued fields.

For a geometric interpretation of the results of this paper we mention, that if the valuation $| \,\ |$ on $K$ is  complete, then any convergent power series 
 $f= \sum c_\alpha x^\alpha \in K\{x\}$, $\|f\|_{\eps}< \infty$, defines a continuous (even analytic) function $f: P^n_\eps \to K, \ f(a) = \sum c_\alpha a^\alpha$, with $P^n_\eps$ the policylinder
 $$P^n_\eps: =
\{a=(a_1,...,a_n) \in K^n \mid |a_i| < \eps_i, i =1,...,n\}.$$
 {However, not every function $f$ that is analytic in $P^n_\eps$ satisfies $\|f\|_{\beps}< \infty$. E.g.  $f = \frac{1}{1-x} = \sum_0^\infty x^\nu$ is analytic in $P^1_1$ (for $K=\C$) but $\|f\|_{1}= \infty$
(see \cite[§1.1, Bemerkung]{GR71}). }\\

We extend $\| \,\ \|_{\eps}$ to $K[[x]]^N$ by setting 
$$\|f\|_{\eps}:= \sum_{i=1}^N\|f_i\|_{\eps}$$
for $f=(f_1,\ldots,f_N)\in K[[x ]]^N=\sum_{i=1}^N K[[x]]e_i$.
Let 
$$B_{\eps}^N:=\{f\in K[[x]]^N \text{  } | \text{ } \|f\|_{\eps}<\infty\}.$$
Then $\| \ \|_{\eps}$ 
satisfies  $ \|f+g\|_{\eps}\le  \|f\|_{\eps}+ \|g\|_{\eps}$ and for 
$f\in K[[x]]$ and $g \in K[[x]]^N$ we have
$ \|fg\|_{\eps}\le  \|f\|_{\eps} \|g\|_{\eps}$.
 If $\rho=(\rho_1,\ldots,\rho_n)\leq\eps=(\eps_1,\ldots,\eps_n)$ then\,\footnote{$\rho\leq\eps$ means $\rho_i\leq\eps_i$ for all $i$.} $B_\eps^N\subseteq B_\rho^N$
and we get
$K\{x\}^N=\bigcup_{\eps\in\R^n_{>0}} B_{\eps}^N$. 

 \begin{Remark}\label{rm.rvf}{\em
\begin{enumerate}
\item [(1)] We allow the trivial valuation  ($|0| =0, |a| =1$ if $a\ne 0$) and then $K\{x\}=K[[x]]$. Hence, any result in this article formulated for $K\{x\}$ holds for  $K[[x]]$, the case which was treated in \cite[Chapter 3.1]{GLS25}.
Finite fields permit only the trivial valuation. The valuation may be non-Archimedian, satisfying the stronger triangel inequality  $|x+y|\le \max(|x|,|y|)$.
\item  [(2)] Examples of a complete real valued field are $\C$  
resp. $\R$ with the usual absolute value as valuation.  {In fact, every field with a non-trivial Archimedean valuation is isometrically isomorphic to a subfield of either $\R$ or $\C$ with the usual valuation ("Ostrowski's 1st Theorem" \cite{Os16}).}

The usual absolute value on the field of rational numbers makes $\Q$ a real valued field, which is not complete (but it is quasi-complete\,\footnote{$K$ is called {\em quasi-complete} if the completion $\bar K$  of $K$ is a separable field extension of $K$. In characteristic 0 every real valued field is already quasi-complete.} since its characteristic is 0).
\item  [(3)] Another example is given by the field of $p$-adic numbers.
Let \mbox{$p$} be a prime number, then the map
$ v_p:\Z\setminus\{0\} \to \R_{>0}\,,\ a \mapsto p^{-m}, \text{ with }
m:=\max\{k\in \N  \mid p^k \text{ divides }a\} $
extends to a unique real  {(non-Archimedian)} valuation of $\Q$. With this
valuation, $\Q$ is a real valued field that is not complete but quasi-complete.
 The completion $\Q_p= \bar \Q$ with respect to  $v_p$ is called the
field of {\em $p$-adic numbers}.
 {For a generalization see (6).\\
By "Ostrowski’s 2nd Theorem" (also in \cite{Os16}) every non-trivial absolute value
on $\Q$ is equivalent to either the usual absolute value or to $v_p$ for some prime $p$.}
\item  [(4)]  Let $F$ be a field and $K=F(T)$ the quotient field of $F[T]$. If $ f/g \in K\smallsetminus \{0\}$ then
$\ord(f/g) = \ord(f)-\ord(g) \in \Z$.  {We set
$|f/g| = e^{-\ord(f/g)}$ (any $c>1$ instead of $e$ could be used) and $|0|=0$. Then $| \ |$ is a non-Archimedian real valuation on $K$, which is not complete.
The completion of $F(T)$ is  the field of {\em formal Laurent series} $K=F((T))$, which is the quotient field of $F[[T]]$, with valuation $|c| = e^{-m}$ if $c=\sum_{i=m}^\infty a_iT^i, m \in \Z, a_m\ne0$. }

\item  [(5)] Another examples of a real valued field  is the field of {\em Puiseux series}  $K=\bigcup_{n=1}^\infty F((T^{\frac{1}{n}}))$,  with valuation 
$|c| = e^{-m/n}$ 
if $c=\sum_{i=m}^\infty a_iT^{\frac{i}{n}}, a_m\ne0$.
 {If $F$ is algebraically closed and of characteristic $0$, then $K$ is the algebraic closure of $F((T))$.
$K$ is not complete, its completion is also of interest and leads to larger fields (e.g. Levi-Civita and Hahn series).}

\item  [(6)]   {$\mathbb{Q}_p$ is an example of a non‑Archimedean local field  (i.e. complete, locally compact field with a non-Archimedian nontrivial valuation whose residue field is finite). Beyond $\mathbb{Q}_p$ there are many other complete non-Archimedean fields. Non‑Archimedean local fields are central objects  of modern theories in number theory, representation
theory, and non‑archimedean geometry, with applications ranging from automorphic forms and local $L$‑functions to rigid analytic geometry and the theory
of $D$‑modules in non‑Archimedean settings.}
\end{enumerate}
}
\end{Remark}

We introduce in Section \ref{sec.invariants} the basic invariants for power series $f \in K\{x\}$,  which generalize 
the classical ones for complex analytic singularities to analytic singularities in arbitrary characteristic. Moreover, using the so called "Lefschetz principle", we show that some results known for $\C$ hold also for real valued fields of characteristic zero.
We discuss then finite determinacy  for power series in $f \in K\{x\}$ with respect to right and contact equivalence  (implying that the truncation up to some finite degree determines its equivalence class), which is indispensable for the classification of singularities. This is used in the last two sections, when we classify singularities in positive characteristic. Finally we prove the Mather-Yau theorem for $f\in K\{x\}$ with $K$ real valued.

In Section \ref{sec.split} we give a proof of the splitting lemma for power series in $K\{x\}$. The splitting lemma is of great importance for the classification of singularities. It says that a power series $f$  splits (up to a coordinate change) as the sum of a quadratic part $f^{(2)}$ plus a series $g$ of order at least three, where $f^{(2)}$ and $g$ depend on disjoint sets of variables.  
Moreover, the series $g$, the "residual part", is unique up to a  change of coordinates.
The original proof (for diﬀerentiable functions) is due to Ren\'e Thom, who used it for his classification of the 7 elementary catastrophes (see his book \cite[§ 5.2D]{Th75}). Thom's proof (for differentiable functions), especially for the uniqueness of the residual part, is fairly long and complicated, while  our proof is quite simple (but with some tricky substitutions).

The previous results are used  in Section \ref{sec.class1} resp. Section \ref{sec.class2}, which culminates in the classification of simple hypersurface singularities in positive characteristic with respect to contact equivalence resp. right equivalence. While the classification with respect to contact equivalence is similar to that in characteristic 0, the classification with respect to right equivalence is fundamentally different.\\

{\bf Acknowledgements:} I like to thank Gerhard Pfister for useful discussions on the contents of this paper. Several results are part of our joint articles \cite{GP25} and \cite{GP25a}.  {Many thanks also to the referee for the careful proofreading.}
Finally, I would also like to thank the Mathematical Institute of the University of Freiburg for its hospitality. 

 {Last but not least, I like to thank Wolfgang Ebeling for many years of fruitful mathematical discussions and long-standing friendship.}

\section{Invariants of Hypersurface Singularities}\label{sec.invariants}
We keep the notations from the Introduction, with $\fm =\x =\langle x_1,...,x_n\rangle$ denoting the maximal ideal of $K\{x\}$.
Given a power series $f\in  K\{x\}$, we  call $f$ or the analytic $K$-algebra
$K\{x\}/\langle f\rangle$
a \emph{hypersurface singularity} if $f \in \fm, f\ne 0$.

The two most important equivalence relations for power series are right equivalence and contact equivalence.

\begin{Definition}\label{def:equiv}
Let $f,g  \in K\{x\}$ and  $k$ a positive integer.
\begin{enumerate}
\item $f$ is {\em right equivalent}
 to $g  \ (f\overset{r}{\sim} g)$ iff
 there exists a $\varphi \in \Aut(K\{x\})$ such that  $f= \varphi(g)$, 
i.e. $f$ and $g$ differ by an analytic coordinate change $\varphi$ with $\varphi(x_i)=\varphi_i \in K\{x\}$ and $f=g(\varphi_1, \cdots, \varphi_n).$

\item $f$ is {\em  contact equivalent} to $g \ (f \overset{c}{\sim} g)$ iff
 there exists a  $\varphi \in \Aut (K\{x\})$ and a unit $u \in K\{x\}^\ast$ such that $f=u\cdot\varphi(g)$, 
i.e. the analytic $K$--algebras $K\{x\}/\langle f\rangle$ and $K\{x\}/\langle g\rangle$ are isomorphic.

\item We call $f$ {\em right $k$-determined} (resp. {\em contact $k$-determined}) if $f \rsim g$ (resp. $f \csim g$) for every  $g \in K\{x\}$ with $f-g \in \fm^{k+1}$ or, equivalently, if $\jet_k(f) = \jet_k(g)$. Here $\jet_k(f)$,  the image of $f$ in $K\{x\}/\fm^{k+1}$, denotes the {\em $k$-jet} of $f$, identified with the power series  of $f$ up to degree $\le k$. 

In both cases we say that $f$ is \emph{finitely determined} if it is $k$-determined for some integer $k>0$, and we call the least such $k$ the \emph{determinacy} of $f$. 
\end{enumerate}
\end{Definition}

\begin{Definition}\label{def:mu_pos}
For  $ f\in K\{x\}$ we set
\begin{enumerate}
\item 
 $$
 \begin{array} {llll}  
  j(f) & :=& \langle\frac {\partial f}{\partial x_1},...,\frac {\partial f}{\partial x_n}\rangle &\text{ \emph{the Jacobian ideal}},\\
tj(f)\rangle & :=& \langle f, \frac {\partial f}{\partial x_1},...,\frac {\partial f}{\partial x_n}\rangle &\text{ \emph{the Tjurina ideal}},
\end{array}  
$$

and call the associated algebras  {
 $$
 \begin{array} {llll}  
 M_f & :=& K\{x\}/j(f)
   &\text{ the \emph{Milnor algebra}, and  }\\
   T_f&:=&K\{x\}/\langle f, j(f) \rangle
   &\text{ the \emph{Tjurina algebra}  of } f.
\end{array}  
 $$ }
   Their dimensions are
$$
 \begin{array} {lllll}  
  \mu(f) & :=& \dim_K K\{x\}/ \langle\frac {\partial f}{\partial x_1},...,\frac {\partial f}{\partial x_n}\rangle\index{Milnor!number} &\text{ \emph{the Milnor number of }} f,\\
 \tau(f) & :=& \dim_K K\{x\}/\langle f, \frac {\partial f}{\partial x_1},...,\frac {\partial f}{\partial x_n}\rangle \index{Tjurina!number}&\text{ \emph{the Tjurina number of }} f.
 \index{$\mu$@$\mu(f)$}\index{$\tau$@$\tau(f)$}
\end{array}  
$$
\index{$mu$@$\mu(f)$}\index{$tau$@$\tau(f)$}  
\item We say that $f$ has an
   \emph{isolated singularity}\index{singularity!isolated} if $\mu(f)<\infty$ (which is 
   equivalent to the existence of a positive integer $k$ such that
   $\fm^k\subseteq j(f)$),  and we
call $K\{x\}/\langle f\rangle$  an \emph{isolated (hypersurface) singularity}\index{singularity!isolated!hypersurface} if
   $\tau(f)<\infty$ (equivalently, if there is a positive integer $k$
   such that $\fm^k\subseteq tj(f)$).  We say also that $f$ {\em defines an isolated singularity}, meaning that $\mu(f)<\infty$ if we consider right equivalence  resp.  $\tau(f)<\infty$ if we consider contact equivalence.
   \end{enumerate}
   \end{Definition}

   These definitions coincide with the complex analytic ones. It is straight forward to see
   that for an automorphism $\varphi\in\Aut(K\{x\})$ and a unit
   $u\in K\{x\}^*$ we have {
 $$
\begin{array}{llll}
    j\big(\varphi(f)\big)&=&\varphi\big(j(f)\big), \\
     tj\big(u\varphi(f)\big)&=&\varphi(tj(f)\big).
 \end{array}
 $$}
   In particular, the Milnor number is \emph{invariant} under right equivalence and
   the Tjurina number is \emph{invariant} under contact equivalence (see \cite[Lemma  1.2.10]{GLS25}). 
   
   It is a non-trivial theorem, using methods from complex analysis, which
   cannot be extended to other fields, that for $f \in \C\{x\}$  the
   Milnor number of a complex analytic hypersurface is indeed invariant under contact
   equivalence (see \cite[p.~262]{Gr75}), and it is even a topological
   invariant (see \cite{LeR}). Using the 
   {\em Lefschetz principle}\index{Lefschetz principle}  the result for contact equivalence can be 
   generalised to arbitrary  real valued fields of
   characteristic zero.  {This was first proved in \cite[Theorem 1]{BGM12} for formal power series. For convergent power series we need an extra argument, which is covered by Lemma \ref{lem.lefschetz}.}
   
 \begin{Lemma}\label{lem.lefschetz}\text{  }
\begin{enumerate} 
\item [(1)]  Let $K$ be a real valued field and  $I\subseteq K\{x\}$ an ideal. Then   
$$\dim_K K\{x\}/I =\dim_K K[[x]]/I K[[x]].$$ 
In particular,
$\dim_K K\{x\}/I< \infty \iff \dim_K K[[x]]/I K[[x]]< \infty $ and if this holds, then $K\{x\}/I = K[[x]]/I K[[x]]$.
\item [(2)]  Let $K\subseteq L$ be a field extension of arbitrary fields and $I\subseteq K[[x]]$ an ideal. Then 
$$\dim_K K[[x]]/I=\dim_{L} L[[x]]/IL[[x]].$$ 
In particular
 $\dim_K K[[x]]/I < \infty \iff \dim_L L[[x]]/I L[[x]]< \infty $ 
  and then we have  $(K[[x]]/I)\otimes_K L=L[[x]]/I L[[x]].$
\end{enumerate} 
If $K\subseteq L$ are real valued fields (with not necessarily compatible valuations) and $I\subseteq K\{x\}$. Then we deduce from (1) and (2) that 
$\dim_K K\{x\}/I=\dim_{L} L\{x\}/IL\{x\}.$ 
\end{Lemma}
   
   \begin{proof}
(1) We have $A:=K\{x\}/I \subseteq \hat A =   K[[x]]/I K[[x]] = \lim\limits_{\longleftarrow} K\{x\}/I+\x^m$,  and hence 
$\dim_K \hat A  < \infty \Longrightarrow \dim_K A < \infty$.
Conversely, if $\dim_K A < \infty$ then $\x^m \subseteq I$ for  $m \ge$ some $m_0$. This implies $\hat A =A$ and the claim follows.

(2) Denote by $\x$ and $\bar \x$,  the ideals generated by $x_1,...,x_n$ in $K[[x]]$ and  $K[x]$ respectively.

For $f \in K[[x]]$ let $\jet_l(f) \in K[x]$ be the $l$-jet of $f$. Then
$f = \jet_{l-1}(f) +h$, with $h \in \x^l K[[x]]$. We set 
$$\jet_l(I):=  \langle\jet_l(f) \mid f \in I \rangle_ {K[x]} \subseteq K[x].$$
It is easy to see that $I + \x^l =\jet_{l-1}(I)K[[x]] + \x^l$. Using this we get
 \begin{align*}
K[[x]]/I+\x^m = & (K[[x]]/\x^m)/(I+\x^m)/\x^m \\
    =& (K[x]/\bar \x^m)/(\jet_{m-1}(I)+\bar\x^m)/\bar\x^m\\
    =& K[x]/ \jet_{m-1}(I) + \bar \x^m.
\end{align*}
Analogous formulas hold for $L[[x]]$ instead of $K[[x]]$. It follows that 
 \begin{align*}
(K[[x]]/I+\x^m)\otimes_K L
  =&(K[x]/ \jet_{m-1}(I) + \bar \x^m)\otimes_K L\\
 =& L[x]/ \jet_{m-1}(I)L[x] +  \x^m L[x]\\
  = & L[[x]]/ IL[[x]] +  \x^m L[[x]]. 
\end{align*}
If  $\dim_K K[[x]]/I <\infty$, then $\x^m \subseteq I$ for some $m$ and hence $\x^m L[[x]] \subseteq I L[[x]]$, showing 
$(K[[x]]/I)\otimes_K L =  L[[x]]/ IL[[x]]$ and
$\dim_L L[[x]]/I L[[x]] = \dim_K K[[x]]/I < \infty $.

If  $\dim_L L[[x]]/IL[[x]] <\infty$ then $\x^l L[[x]] \subseteq IL[[x]]$ for all
$l\ge m$ and it follows that $\dim_L L[[x]]/IL[[x]] = \dim_L (K[[x]]/I+\x^l)\otimes_K L = \dim_K(K[[x]]/I+\x^l)$ for $l\ge m$.
With $A:= K[[x]]/I$ and $\fm = I+\x/I$ its maximal ideal, it follows that
the surjection
$A/\fm^ {m+1} \to A/\fm^ {m}$ is an isomorphism and hence $\fm^ {m+1} = \fm^ {m}$. By Nakayama's lemma $\fm^ {m} =0 $, i.e., 
$\x^m \subseteq I$, showing  $\dim_K K[[x]]/I <\infty$.
 \end{proof}

  \begin{Proposition}\label{prop:lefschetz1}
     Let $K$ be a real valued field of characteristic zero
     and $f,g\in K\{x\}$. 
     If $ f \overset{c}{\sim}\ g$, then $\mu(f)=\mu(g)$.
   \end{Proposition}
   
   \begin{proof}
     Since the Milnor number is invariant under right equivalence, it
     suffices to show that $\mu(f)=\mu(u\cdot f)$ for any unit
     $u\in K\{x\}^*$. If $A$ denotes the subset of $K$
     containing the coefficients of $u$, $f$ and all partial
     derivatives $\frac{\partial f}{\partial x_i}$ of $f$, then $A$ is at most countable  infinite and $\{f, u, \frac{\partial f}{\partial x_i} \} \subseteq \Q(A)\{x\}$.
     Since $\Char(K)=0$ and since $\Q\subset\C$ is a field
     extension of uncountable transcendence degree 
     the field $\Q(A)$ is isomorphic to a subfield ${L}$
     of the field $\C$ of complex numbers. The valuation on $L$, which is inherited from $\Q(A)$,  may however not be the usual valuation on $\C$. 
 {Nevertheless, by  Lemma \ref{lem.lefschetz}  we have
$\dim_{L}({L}\{x\}/j(f))=\dim_\C(\C\{x\}/j(f))$ and
$\dim_{L}({L}\{x\}/j(uf))=\dim_\C(\C\{x\}/j(uf))$. Since the right hand sides coincide by  \cite{Gr75}, the proposition follows.}
 \end{proof}

  In positive characteristic this result does not hold;
   e.g.\ if $\Char(K)=p>0$ and $f=x^p+y^{p-1}$, then $\mu(f)=\infty$
   while the contact equivalent series $g=(1+x)\cdot f$ has Milnor
   number $\mu(g)=p\cdot (p-2)$.\\

In complex singularity theory  the Milnor number of a power series is finite if and only if the Tjurina
   number is so (see \cite[Lemma 1.2.3]{GLS25}). This fact can also be
   generalised to arbitrary fields of \emph{characteristic zero}.

   \begin{Theorem}\label{thm:lefschetz2}
     Let $ K$ be a field of characteristic zero
     and $f\in K\{x\}$.
     Then $\mu(f)<\infty$ if and only if $\tau(f)<\infty$.    
   \end{Theorem}
   \begin{proof}
     Let $A$ be the set of
     coefficients of $f$ and all its partial derivatives. We proceed as in the     proof of Proposition \ref{prop:lefschetz1}: Then the field $\Q(A)$ is     isomorphic to a subfield ${L}$ of $\C$
     and we get 
     \begin{displaymath}
       \mu(f)=\dim_{L}({L}\{x\}/j(f))=\dim_\C(\C\{x\}/j(f)),
     \end{displaymath}
     \begin{displaymath}
      \tau(f)=\dim_{L}(L\{x\}/tj(f)) =\dim_\C(\C\{x\}/tj(f)).
     \end{displaymath}
     Using the result for $\C\{x\}$,  we get that $\tau(f)$ is finite if and only if $\mu(f)$ is finite.     
 \end{proof}

   For fields of positive characteristic this is false. For $f=x^p+y^{p-1}$ we have
   $\tau(f)=p\cdot (p-2)$ while $\mu(f)=\infty$. \\
     
We discuss now finite determinacy of hypersurfaces w.r.t. right and contact equivalence. 
For the classification of  power series with
   respect to right respectively 
   contact equivalence a first important step, 
   from a theoretical point of view as well as from a practical one, is
   to know that the equivalence class is determined by a finite number
   of terms of the power series $f$ and to find the (smallest) corresponding degree bound (see also \cite[ Sections 1.2.2 and 3.1.2]{GLS25}).  
 {In characteristic 0 or big characteristic we have the following:
\begin{Theorem}  \label{rem:det-char0} 
Let $K$ be a real valued field of characterisit 0 or of positive characteristic $p \ge k +2 -\ord(f)$ for some integer $k$.
For $f\in \fm \subseteq K\{x\}$,  $f$ is finitely determined w.r.t.\ right or contact equivalence if and only if $f$
defines an isolated singularity.  In fact, we have:
\begin{enumerate}
\item [(a)] If $\fm^{k+1}\subset\fm^2\cdot j(f)$, then $f$ is right $k$-determined. \\
In particular, the right determinacy is at most
   $\mu(f)+1$.
\item [(b)]  If $\fm^{k+1}\subseteq\fm\cdot\langle f\rangle+\fm^2\cdot j(f)$, then $f$ is  contact $k$-determined. \\ In particular,  the contact determinacy is at most  $\tau(f)+1$.
 \end{enumerate}
\end{Theorem}
\begin{proof}
The proof of \cite[Theorem 1.2.23]{GLS25} can be adapted to work for any field of characteristic $0$, as noted in \cite[Remark 1]{BGM12}. Another proof (for a more general statement) is given by Kerner \cite[Corollary 5.5.]{Ker}, which implies the above bounds  not only in characteristic  $0$ but also in positive characteristic if $\Char(K) \ge k +2 -\ord(f)$.
\end{proof}
}

\begin{Example} \label{ex.detbound}{\em  For positive characteristic these bounds do not hold.\\
Consider the power series
       $f=y^2+x^3y\in K[[x,y]]$, $\Char(K)=2$. Then
        $\langle f\rangle+ j(f)
     =\langle y^2,x^2y,x^3\rangle$ and thus $\tau(f)=5$. 
 Moreover, we have
       \begin{displaymath}
  \fm^5\subseteq\fm\cdot\langle f\rangle+\fm^2\cdot j(f) = \langle xy2,y3,x5,x4y\rangle 
        \end{displaymath}
       and if (b) would hold then $f$ would be contact
       $4$-determined. However, $f$ is analytically reducible
        while
       $f+x^5$ is irreducible as can be checked by the procedure
       \texttt{is\_irred} in \textsc{Singular} 
       \cite{DGPS}: 
\begin{verbatim}
> LIB "hnoether.lib";
> ring r=2,(x,y),ds;
> poly f = y2 + x3y;
> is_irred(f);
0
> poly g = f + x5;
> is_irred(g);
1   
\end{verbatim}  
Hence  $f\not\sim_c f+x^5$, and $f$ is not  contact $4$-determined.        
Theorem~\ref{thm:finitedeterminacybound} asserts that $f$ is actually 
       $6$-determined, i.e.\ our result is sharp in this example.
}
\end{Example}

In positive characteristic higher determinacy bounds  were proved 
 in \cite[Theorem 3]{BGM12}, and in a slightly weaker form for contact equivalence in   \cite[Lemma 2.6]{GrKr90} for formal power series. 
  Before we generalize these bounds to convergent power series over real valued fields in Theorem \ref{thm:finitedeterminacybound}, we mention the following important approximation theorem of Grauert, which is used in the proof. \\
 
  We consider four sets of variables, $x = (x_1, \ldots, x_n)$, $s =
  (s_1, \ldots, s_l)$, $Y = (Y_1, \ldots , Y_p)$ and $Z = (Z_1,
  \ldots, Z_q)$. 

\begin{Definition}\label{def.sol}
Let $K$ be a real valued field, $I \subseteq K\{s\}$ an ideal and $\fm:=\langle s_1,\ldots,s_l\rangle$ the maximal ideal of  $K\{s\}$.  Consider an element $F = (F_1, \ldots, F_N)  \in K\{x,s,Y,Z\}^N$.

(1) A {\em solution of order $e$} of the equation $F \equiv 0 \; {\rm mod} \;I$
  is a pair \\
\noindent $(y, z) \in K[s]^p \times K\{x\}[s]^q$ such that
\[
F(x,s, y(s), z(x,s)) \equiv 0 \quad {\rm mod} \; (I+ \fm^{e+1}) \cdot
K\{x,s\}^N.
\]

(2) An {\em analytic solution} of the equation $ F \equiv 0$ mod $I$ is a
  pair $(y, z) \in K\{s\}^p \times K\{x,s\}^q$ 
  such that
\[
F(x,s, y(s), z(x,s)) \equiv 0 \quad {\rm mod} \; I\cdot K\{x,s\}^N, 
\]
\end{Definition}
We can now formulate the approximation theorem.
 
\begin{Theorem}[Grauert's Approximation Theorem for real valued fields]
\label{grauertapprox}
Let $K$ be a real valued field, $I \subseteq K\{s\}$ an ideal and  $F = (F_1, \ldots, F_N)  \in K\{x,s,Y,Z\}^N$.
Let $e_0 \in \N$ and suppose that the system of equations  {
$$F(x,s,Y,Z) \equiv 0 \; {\rm mod} \; I$$
has
a solution $(Y,Z)=(y^{(e_0)}, z^{(e_0)})$ of order $e_0$.}

 Suppose, moreover,
that for all $ e \ge e_0$ every solution 
$(y^{(e)},z^{(e)})$ of order $e$  with
$y^{(e)} \equiv y^{(e_0)} \mod
{\fm}^{e_0+1}$, 
and $z^{(e)} \equiv z^{(e_0)} \mod {\fm}^{e_0+1}$,
extends to a solution  $(y^{(e)} +
u^{(e)}, z^{(e)} + v^{(e)})$ of order $e+1$,
 with
$u^{(e)}\in K[s]^p$ and $v^{(e)} \in K\{x\}[s]^q$ homogeneous of
degree $e+1$ in $s$.

Then the system of equations $F \equiv 0$ mod $I$ has
an analytic solution $(y(s), z(x,s))$, with $y \equiv y^{(e_0)}$ mod
${\fm}^{e_0+1}$, and $z \equiv z^{(e_0)}$ mod ${\fm}^{e_0+1}$.
\end{Theorem} 

The proof of this theorem  was first given by Grauert for 
$K\{x\} = \C\{x\}$  in his fundamental Inventiones paper \cite{Gr72}, where he proved the existence of a semiuniversal deformation of any complex analytic isolated singularity. The theorem was generalized to power series over a real valued field in the above form in 
 \cite{GP25}. The proof is quite complex, so
 we refer to the proof of Theorem 4.3 in  \cite{GP25}.
  
   \begin{Theorem}[ {Determinacy bound in characteristic $\ge 0$}]\label{thm:finitedeterminacybound}
     Let $0\not=f\in\fm^2 \subseteq K\{x\}$, $K$ an arbitrary real valued field, and $k\in\N$.
     \begin{enumerate}
     \item If $\fm^{k+2}\subseteq\fm^2\cdot j(f)$, then $f$ is
       right $(2k-\ord(f)+2)$-determined. 
     \item If $\fm^{k+2}\subseteq\fm\cdot\langle f\rangle+\fm^2\cdot j(f)$, then $f$ is  contact $(2k-\ord(f)+2)$-determined. 
     \end{enumerate}     
 If $0\not=f\in \fm \smallsetminus \fm^2$ then $\ord(f) =1$ and 
 $f$ is right 1-determined with $f\overset{r }{\sim} x_1$ by the implicit function theorem.
   \end{Theorem}
   
\begin{proof} The proof was first given in  \cite[Theorem 3]{BGM12} for formal power series. We give here a (shorter) proof for $K\{x\}$ using of Grauert's approximation theorem.
We prove only statement 1., the proof of statement 2. is similar.

Let $N\geq 2k-\ord(f)+2=:e_0$ and assume that $f-g\in \fm^{N+1}$. We have to prove that there exists an automorphism $\varphi=(\varphi_1,\ldots,\varphi_n)$
of $K\{x\}$ such that $f(\varphi)=g$. We consider the equation 
$$F(x,y)=f(y)-g(x)=0, \ y=(y_1,\ldots,y_n),$$
and want to apply Theorem \ref{grauertapprox} to the equation $F=0$. Since $f-g\in \fm^{e_0+1}$,
this equation has the solution $\varphi^{(e_0)}(x_i)=x_i$ of order  $e_0$:
$$F(x,\varphi^{(e_0)})=0 \mod \fm^{e_0+1}.$$
Let $\varphi^{(e)}$ be a solution of $F(x,y)=0$ of order $e \ge e_0$ such that $\varphi^{(e)}\equiv \varphi^{(e_0)} \mod \fm^{e_0+1}$, that is, 
$\varphi^{(e)}(x_i) = x_i + h_i$ with $h_i \in \fm^{e_0+1}. $
By assumption we have 
$\fm^{k+2}\subseteq \fm^2j(f)$,  {hence $k\ge \ord(f) -1$,} and
$F(x,\varphi^{(e)})\subseteq \fm^{e+1}$.
Since $\fm^{e+1} = \fm^{e-k-1}\fm^{k+2} \subseteq \fm^{e-k+1}j(f)$  {and  $e-k-1\ge e_0-k-1=k-\ord(f) +1\ge 0$,}
we can write
$$F(x,\varphi^{(e)})=f(\varphi^{(e)})-g(x)=\sum_{i=1}^nb_i(x)\frac{\partial f}{\partial x_i}(x)$$
with $b_i\in \fm^{e-k+1}$. Now define
$\varphi^{(e+1)}$ by $\varphi^{(e+1)}_i=\varphi^{(e)}_i-b_i$. We obtain by Taylor's formula (Remark \ref{rm.taylor})
$$f(\varphi^{(e+1)})=f(\varphi^{(e)}-b)=f(\varphi^{(e)})-\sum_{i=1}^nb_i(x)\frac{\partial f}{\partial x_i}(\varphi^{(e)})+H(x),$$
with 
$$\ord(H) \ge \ord(f) + 2( \ord(b) -1) \ge \ord(f) + 2( e-k)
\ge e + (e_0-2k +\ord(f)) = e+2.$$
Moreover
$ \frac{\partial f}{\partial x_i}(\varphi^{(e)}) =   \frac{\partial f}{\partial x_i}(x+h) = \frac{\partial f}{\partial x_i}(x) +
 \sum_{j=1}^n h_j \frac{\partial^2 f}{\partial x_j\partial x_i}(x) + G(x)$.
 We can estimate 
(using $\ord (h_j) \ge e_0+1$ and $\ord (b_i)\ge e-k+1$):
$$\ord (b_iG) 
 \ge (e-k+1)+ \ord(f)-1 + 2e_0  =e +(\ord(f) + e_0-k) +e_0= e+ k+2+e_0 \ge e+2$$ 
 and
$\ord(b_ih_j \frac{\partial^2 f}{\partial x_j\partial x_i}(x) \ge  e+k+2 \ge e+2$.
 
This implies $f(\varphi^{(e+1)})=f(\varphi^{(e)})-\sum_{i=1}^nb_i(x)\frac{\partial f}{\partial x_i}(x)$ + terms of order $\ge e+2$ and thus
$F(x,\varphi^{(e+1)})=0 \mod \fm^{e+2}$, i.e.,
$\varphi^{e+1}$ is an extension of the solution $\varphi^{(e)}$ of order $e+1$. We can apply
Grauert's Approximation Theorem \ref{grauertapprox} and obtain a convergent solution $\varphi$ of $F=0$. 
This proves statement 1.
\end{proof}

\begin{Remark}\label{rm.taylor}{\em
 {(1) For an improvement and far reaching generalizations of the determinacy bounds (also for right-left equivalence) see the work of Kerner 
\cite{Ker}, with background provided in \cite{BGK}.}

(2) We use (the replacement of) the Taylor series in positive characteristics in the following form:\\
Let $f(x)= \sum\limits_{\left| {{\alpha}}  \right| \ge \ord(f)} {c_\alpha  x^{{\alpha}}}\in K\{x\}$, $x=(x_1,...,x_s)$ and ${z}=(z_1,...,z_s)$ new variables. Then  
\begin{align*}
      f({x+z}) =& \sum_{|\alpha|\ge \ord(f)}c_\alpha(x+z)^\alpha
	=f({ x})+
	\sum\limits_{\nu = 1}^s {\frac{{\partial f({x})}}{{\partial x_\nu }}\cdot z_\nu }  + H,\\
H=&	\sum\limits_{\left| \alpha  \right| \ge \ord(f)} {c_\alpha\cdot  
\Big ( 
{\sum\limits_{\left| \gamma  \right| \ge 2 \atop \gamma  \le \alpha } 
{\binom {\alpha _1 } { \gamma _1} \cdot\ldots\cdot
\binom {\alpha _s } { \gamma _s}{x}^{\alpha  - \gamma }{z}^\gamma} }
\Big )}.
\end{align*}
\noindent
If $z=z(x) \in K\{x\}$ then $\ord H(x) \ge  \ord(f) + |\gamma|  (\ord z(x)-1) \ge \ord(f) + 2  (\ord z(x)-1).$
$\gamma\le \alpha$ means that $\gamma_\nu\le\alpha_\nu$ for all $\nu$, and for $k\in \Z$
we have $k{ x}^{\alpha  - \gamma }{ z}^\gamma= 0$ if   $p\mid k$ and  $k {(\modular\hskip 3pt p)}{ x}^{\alpha  - \gamma }{ z}^\gamma$ if  $p\nmid k$. 
}
\end{Remark}

\begin{Example}   \label{ex:determ}  {\em
In concrete examples the integers $k$ in
       Theorem~\ref{thm:finitedeterminacybound} can be computed in 
       \textsc{Singular} (\cite{DGPS}) with the aid of the procedure
       \texttt{highcorner}. If we apply \texttt{highcorner} to a
       standard basis of the ideal $J=\fm^2\cdot j(f)$ resp.\ $J=\fm\cdot\langle
       f\rangle+\fm^2\cdot j(f)$ with respect to some local degree
       ordering the result will be a monomial 
       $\bx^\alpha$, and then $k=\deg(\bx^\alpha)-1$
        {satifies $\fm^{k+1}\subseteq J$.} 
       E.g.\ for
       $f=y^8+x^8y^4+x^{23}$ and $\Char(K)=3$ the following
       \textsc{Singular} computation shows that
       $k=\deg(x^{22}y^2)-1=23$ and $f$ is at least contact $40$-determined by Theorem \ref{thm:finitedeterminacybound}
\begin{verbatim}
> ring r=3,(x,y),ds;
> poly f=y8+x8y4+x23;
> ideal I=maxideal(1)*f+maxideal(2)*jacob(f);
> I=std(I);
> highcorner(I);
x22y2
\end{verbatim}   
In characteristic 0 we get the same  \texttt{highcorner} and hence $f$ is 24-determined  {by the above bounds for characteristic 0}.
On the other hand,  in characteristic 2  we get  $x^{21}y^7$ as highcorner and hence $f$ is 48-determined. This is for contact equivalence. For right equivalence we get $x^{29}y^2$ as 
 \texttt{highcorner} in characteristic 0 and characteristic 3 and thus we get 31 resp. 56 as bound for right-determinacy in characteristic 0 resp. 3.
 In characterisitc 2 $f$ has no isolated singularity.
 }
\end{Example}

   If $\mu(f)<\infty$ then $\fm^{\mu(f)}(K\{x\}/j(f))=0$ and hence $\fm^{\mu(f)} \subseteq j(f)$. Thus the assumption of Theorem \ref{thm:finitedeterminacybound} in (1) holds for $k=\mu(f)$ and similar in (2) for    $k=\tau(f)$. Therefore we get:
   
\begin{Corollary}\label{cor:finitedeterminacybound}
With $f$ as in Theorem \ref{thm:finitedeterminacybound} we have
     \begin{enumerate}
     \item If $\mu(f)<\infty$, then  $f$ is right
       $(2\mu(f)-\ord(f)+2)$-determined.
     \item If $\tau(f)<\infty$, then  $f$ is contact
      $(2\tau(f)-\ord(f)+2)$-determined. 
     \end{enumerate} 
\noindent
     In particular, if $\ord(f) \geq 2$ then $f$ is right $2\mu(f)$-determined and contact  $2\tau(f)$-determined.    \index{determinacy!bound}
 \end{Corollary}   
      
   We show now that the finite determinacy is equivalent to $\mu(f) < \infty$ for right equivalence and to  $\tau(f) < \infty$  for   contact equivalence, which are two distinct conditions  in positive characteristic.

\begin{Theorem}\label{thm:hypersurface}
Let $K$ be any  {real valued} field and $0\not= f\in \mathfrak{m}\cdot  K\{x\}$. 
\begin{enumerate}
\item $f$ is finitely right determined iff $\mu(f)$ is finite.
\item $f$ is finitely contact determined iff $\tau(f)$ is finite. 
\end{enumerate}
\end{Theorem}

\begin{proof} According to Corollary \ref{cor:finitedeterminacybound} we only need to show that finite determinacy implies that $\mu(f)$ or $\tau(f)$ are finite. The theorem is stated in  \cite[Theorem 4.8]{GPh2} for $K[[x]]$ with a rather short proof. Therefore we prove it here in full detail.\\
By Lemma \ref{lem.lefschetz} we have $\mu(f) = \dim_K K[[x]]/j(f)$ and $\tau(f) = \dim_K K[[x]]/tj(f)$.

1. Let $f$ be right $k$-determined. The easiest way to see that $\mu(f)$ is finite is perhaps by using the theory of Gr\"obner bases.
By finite determinacy we may assume that $f$ is a polynomial of some degree $d$ ($\ge k$).
Set $g_N := x_1^N+...+x_n^N$, $N > d$, and let $ \Char (K) \nmid N$ if $\Char (K) >0$. Then $f\overset{r}{\sim} f+g_N$ and hence $\mu(f) = \mu(f +g_N)$. The Jacobian ideal of $f +g_N$ is
$$j(f +g_N) = \langle N x_1^{N-1}+\frac {\partial f}{\partial x_1},...,N x_n^{N-1}+\frac {\partial f}{\partial x_n}\rangle \subseteq K[x].$$
Now choose a global degree ordering on $K[x]$ (see \cite{GG}). Since the degree of $\frac {\partial f}{\partial x_i}$
is $\le d-1 < N-1$, it follows that $x_i^{N-1}$ is the leading monomial of 
$N x_i^{N-1}+\frac {\partial f}{\partial x_i}$
with respect to a global degree ordering. Hence $x_i^{N-1}$ is  contained in the leading ideal $L(j(f +g_N))$ of $j(f +g_N)$ for $i=1,...,n$ and obviously $\dim_K K[x]/L(j(f +g_N))<\infty$.
By \cite[Corollary 7.5.6]{GG} we have 
$$\dim_K K[x]/j(f +g_N) = \dim_K K[x]/L(j(f +g_N))< \infty.$$
On the other hand 
$$\dim_K K[[x]]/j(f +g_N)K[[x]] \le \dim_K K[x]/j(f +g_N).$$
To see this, we may assume that $K$ is algebraically closed, since passing to the algebraic closure does not change the dimension by Lemma \ref{lem.lefschetz}. Now, for every ideal $I\subseteq K[x]$ with $\dim_KK[x]/I < \infty$
the the variety $V(I)$ consists of finitely many maximal ideals of the form $\langle x_1 -p^i_1,...,x_n-p^i_n\rangle$,
$p^i=(p^i_1,...,p^i_n) \in K^n$, $i=1,...,s$, with, say, $p^1 = 0$. We have (see also \cite[Appendix A.9]{GG})
 {$$ \dim_K K[x]/I = \dim_K K[[x]]/IK[[x]] +\sum^s_{i=2} \dim_K K[[x-p^i]]/IK[[x-p^i]].$$}
Therefore 
$\mu(f) = \dim_K K[[x]]/j(f +g_N)K[[x]] < \infty$.

2. Let $f$ be contact  $k$-determined. We use a different argument to show that $\tau(f) < \infty$. 
 {For every fixed $ 0\ne t \in K$ we have  $tf\overset{c}{\sim} f$ and hence $tf$ is contact $k$-determined.
Choose the polynomial $g_N$ as in 1. and set
$$ f_t := g_N + t f \in K[t]\{x\}.$$
Then $tj(f_t) = \langle\frac {\partial f_t}{\partial x_1},...,\frac {\partial f_t}{\partial x_n}, f_t\rangle \subseteq K\{x\}$ and  
$f_t \overset{c}{\sim} tf \overset{c}{\sim} f$.} We get 
 $$\tau (f) = \tau (f_t) = \dim_K K\{x\}/tj(f_t) \text{ for } t\ne 0.$$ 
By the semicontinuity of $\tau$, Proposition \ref{prop.semcont}
 below,
there exists an open  neighbourhood $U \subseteq K$ of 0 such that  
$$\tau (f_t) \le \tau (f_0) = \tau(g_N) < \infty \text{ for } t \in U.$$
As in 1. we may assume that $K$ is algebraically closed. Then $K$ is infinite and hence $U\smallsetminus \{0\}$ is not empty and for $t \in U\smallsetminus \{0\}$ we have
$\tau(f) = \tau(f_t) < \infty$.
 \end{proof}

\begin{Remark} {\em
 The above results were proved more generally for matrices and ideals of formal power series in  \cite{GPh1} and  \cite{GPh2} (see also \cite[Section 3.1.3]{GLS25}). 
The determinacy bounds of Theorem \ref{thm:finitedeterminacybound} have
been generalized to more classes of singularities and maps and more equivalence relations in a very general context in \cite{BGK} and \cite{Ker}, see also \cite[Appendix A]{GLS25}.
}
\end{Remark}

We mention the following semicontinuity of $\mu$ and $\tau$, which was used above, but is of independent interest. Let $A$ be a Noetherian ring  and $F\in A[[x]]$ a formal power series with coefficients in $A$. For $\fp \in \Spec A$ let $k(\fp) = A_\fp/\fp A_\fp$ be the residue field of the local ring $A_\fp$ and $F(\fp) \in k(\fp)[[x]]$ the image of $F$ in $k(\fp)[[x]]$. Moreover, we have
\begin{center}
$\mu(F(\fp)) = \dim_{k(\fp)} k(\fp)[[x]]/\langle\frac {\partial F(\fp)}{\partial x_1},...,\frac {\partial F(\fp)}{\partial x_n}\rangle$\\
$\tau(F(\fp)) = \dim_{k(\fp)} k(\fp)[[x]]/\langle\frac {\partial F(\fp)}{\partial x_1},...,\frac {\partial F(\fp)}{\partial x_n}, F(\fp)\rangle.$
\end{center}

\begin{Proposition}[Semicontinuity of $\mu$ and $\tau$] 
\label{prop.semcont}
Let $A$ be Noetherian, $F\in A[[x]]$, and  $\fp \in \Spec A$. 
Then $\mu(F(\fp))$ and $\tau(F(\fp))$ are semicontinuous at $\fp \in \Spec A$. That is, there is an open neighbourhood $U \subseteq \Spec A$  of $\fp$ such that 
$\mu(F(\fq))\leq \mu(F(\fp))$   and $\tau(F(\fq))\leq \tau(F(\fp))$ for all  $\fq \in U$.
\end{Proposition}

\begin{proof} The proof is given in \cite[Proposition 3.2.34]{GLS25}, based on \cite[Proposition 3.1]{GP21}. A proof for $A=K[t]$ and for closed points $t\in \Spec A$ is given in \cite[Proposition 3.4]{GPh2}. 
\end {proof}

In particular, if $A=K[t]/I$, $t=(t_1,...,t_k)$, $I$ an ideal and $K$ algebraically closed, each closed point $\fp  \in V(I) \subseteq \Spec A$ is of the form $\langle t-p\rangle,\  p\in K^k$. Then $k(\fp)=A/ \langle t-p\rangle \cong K$ and $F(p):= F(\fp) \in K[[x]]$. If  $K$ is a real valued  field, we write $F \in A\{x\}$ if $F(p)\in K\{x\}$ for each $p\in \Max A$ (the set of maximal ideals of $A$), which we identify with $K^k$. The semicontinuity theorem then says that for each $p \in K^k$ there exists a Zariski open neighbourhood $U\subseteq V(I)$ of $p$ such that 
$$\mu(F(q))\leq \mu(F(p)) \text{ and } \tau(F(q))\leq \tau(F(p)) \text{ for all } q \in U.$$
 {The same result holds for  $A= K\{t\}/I$ and $F\in A\{x\} = K\{x,t\}/ I K\{x,t\}$, $K$ a complete real valued field (e.g. $K=\C$ or $\R$),} with $U=P^K_\eps(p) \cap V(I)$, $P^K_\eps(p)$ the policylinder  $\{ q=(q_1,...,q_n) \in K^k \mid |q_i-p_i]< \eps_i\}$, $\eps_i  >0$ (see the beginning of the Introduction). This follows from 
Proposition \ref{prop.semcont}, since the Milnor and Tjurina numbers do not change if we pass from $A\{x\}$ to $A[[x]]$ (see Lemma \ref{lem.lefschetz}).\\

We complete this section with  generalizations of the Mather-Yau theorem to power series $f$ in $K\{x\}$.
Recall the higher Tjurina and Milnor algebras from  {\cite{GLS25} and} \cite{GPh} for $k\ge0$:
\begin{eqnarray*}
T_k(f)&:=&K\{x\} \big/\big\langle f, \mathfrak{m}^k j(f)\big\rangle \text{ resp.} \\
M_k(f) &:=&K\{x\}/ \mathfrak{m}^k j(f),
\end{eqnarray*}
are called the \textit{$k$-th Tjurina} resp. \textit{$k$-th Milnor algebra} of $f$. For $k=0$ this is the usual Tjurina algebra $T_f=K\{x\}/\langle f, j(f) \rangle$ resp. Milnor algebra $M_f=K\{x\}/j(f)$  from Definition \ref{def:mu_pos}, with dimensions $\tau(f)$ resp. $\mu(f)$.

\begin{Proposition}[Mather-Yau in characteristic 0]
\label{prop.MY0}
Let $K$ be a real valued, algebraically closed field of characteristic 0 and $f, g \in\mathfrak{m}\subseteq K\{x\}$. 
\begin{enumerate}
\item  If $\tau(f)<\infty$ the following are equivalent:\\
i) $f\mathop  \sim \limits^c g$.\\
ii)  For all $k\ge 0$, $T_k(f)\cong T_k(g)$ as $K$-algebras.\\
iii) There is some $k\ge 0$ such that  $T_k(f)\cong T_k(g)$ as $K$-algebras.
\item If $\mu(f)<\infty$ the following are equivalent:\\
i) $f\mathop  \sim \limits^r g$.\\
ii) For all $k\ge 0$, $M_k(f)\cong M_k(g)$ as $K\{t\}$-algebras.\\
iii) There is some $k\ge 0$ such that  $M_k(f)\cong M_k(g)$ as $K\{t\}$-algebras.
\end{enumerate}
\end{Proposition}

\begin{proof}
 {The statement was proved in \cite[Proposition 2.1]{GPh} for $K[[x]]$, $K$ an algebraically closed field  (it is deduced  from  \cite[Theorem 1.2.26]{GLS25} for $\C\{x\}$ by using the "Lefschetz principle"). 
The convergent case follows from Artin's approximation theorem. As an example (and a model for the other cases in this paper where Artin's theorem is used), we write down the argument for contact equivalence explicitly.}

 {
Let $f, g\in \mathfrak{m}\subseteq K\{{x}\}$ be such that the finite dimensional $K$-algebras $T_k(f)$ and $T_k(g)$ are isomorphic for some $k\ge 0$. Then, by \cite[Proposition 2.1]{GPh} there are $\bar\psi\in \Aut(K[[x]])$, $\bar\psi(x_i) = :\bar\psi_i(x_1,...,x_n)$,  and a unit $\bar u\in K[[x]]^{\ast}$, 
 such that 
$ g(x) = \bar u(x) f(\bar\psi_1(x),...,\bar\psi_n(x)).$  
That is, the  equation 
$$F(x,y) = g(x)-y_0 f(y_1,...,y_n)=0,$$
$F(x,y) \in K\{x,y_0,y_1,...,y_n\}$,
has a formal solution $y_0=\bar u(x), y_i = \bar\psi_i(x)$. By Artin's analytic approximation  theorem \cite {Ar68} we get  for any $c>0$ convergent solutions, that is  $u(x), \psi_i(x) \in K\{x\}$, such that $ g(x) -  u(x) f(\psi_1(x),...,\psi_n(x)) =0$. Moreover,  
$u$ und $\psi$ coincide with $\bar u$ and $\bar \psi_i$ up to order $c$. Taking $c\ge 2$, then $u\in K\{x\}^\ast$  is a unit and
$\psi = (\psi_1,...,\psi_n) \in \Aut(K\{x\})$. This proves the theorem.
}
\end{proof}

 {The condition that $K$ is algebraically closed cannot be omitted. Let $a \in K$ , $d \ge 2$ and $\sqrt[d]{a} \notin K$. Then  $T_k(f)=T_k(g)$ for $k\ge 0$ but
$f=x^d-y^d$ is not contact equivalent to $g=x^d-ay^d$. The same applies to right equivalence.}

Proposition \ref{prop.MY0} is wrong in positive characteristic. E.g., take $f=y^2+x^3y$ and $g=f+x^5$, $\Char K=2$. Then $T_1(f)=T_1(g)$ but $f\mathop{\not\sim}\limits^{c} g$. 
 {For Mather-Yau in positive characteristic} we need the higher Tjurina and Milnor algebras.

\begin{Theorem}[Mather-Yau in characteristic $>0$] \label{thm.MYp}
Let $K$ be any real valued field and let $f,g\in K \{x\}$ with $\ord(f) \ge 2$.
\begin{enumerate}
\item If $\tau(f)<\infty$ the following are equivalent:\\
i) $f\mathop  \sim \limits^c g$. \\
ii) $T_k(f)  \cong T_k(g)$ as $K$-algebras for some (equivalently for all) $k\ge2(\tau(f)-\ord(f)+2)$.
\item If $\mu(f)<\infty$ the following are equivalent:\\
i) $f\mathop  \sim \limits^r g$. \\
ii) $M_k(f)  \cong M_k(g)$ as $K\{t\}$-algebras for some (equivalently for all) $k\ge2(\mu(f)-\ord(f)+2)$.\\ Note that $t$ acts on $M_k(f)$ by multiplication with $f$ and on  $M_k(g)$ by multiplication with $g$; an isomorphism $\varphi$ of $K\{t\}$-algebras means    $\varphi(f)=g$.
\end{enumerate}
\end{Theorem}

We refer  to the proof for formal power series given in \cite{GPh}, Theorems 2.2 and 2.4 and Corollaries 2.3 and 2.5, which carries over to $K\{x\}$  {by using Artin's approximation theorem as in the proof of Proposition \ref{prop.MY0}}. We get in particular the simple bounds $k\ge2\tau(f)$ resp. $k\ge2\mu(f)$.

We mention that in \cite[Theorems 2.2, resp. Theorems 2.4]{GPh} the following better bounds for $k$  {in 1. ii) resp 2. ii)} are given:
 $$\mathfrak{m}^{\left\lfloor {\frac{{k + 2\ord(f)}}{2}} \right\rfloor}\subset\mathfrak{m}\left\langle f \right\rangle +\mathfrak{m}^2{{j}}(f)
\text{ resp. } \mathfrak{m}^{\left\lfloor {\frac{{k + 2\ord(f)}}{2}} \right\rfloor}\subset\mathfrak{m}^2{ {j}}(f),$$
\noindent where ${\left\lfloor {\frac{{k + 2s}}{2}} \right\rfloor}$ means the maximal integer which does not exceed $\frac{{k + 2s}}{2}$.


\section{Versal Deformations and the Splitting Lemma}\label{sec.split}

Let $K$ be again a real valued field and $K\{x\}$, $x= (x_1,...,x_n)$, the convergent power series ring over $K$. As an application of the Approximation Theorem \ref{grauertapprox} we prove first (in Theorem \ref{thm.suunf}) the existence of a semiuniversal unfolding of an isolated hypersurface singularity $f \in \fm^2 \subseteq K\{x\}$, $\fm =\langle x\rangle$ the maximal ideal. This, and the determinacy bound for isolated singularities (Theorem \ref{thm:finitedeterminacybound}), is used to prove the splitting lemma for any (not necessarily isolated) $f \in K\{x\}$ (see Theorem \ref{thm.split}).
\medskip

Let us recall the definition of a (semiuni-) versal deformation of an analytic algebra. For a geometric formulation for complex germs see \cite{GLS25}, and \cite{St03} for general cofibred groupoids.

\begin{Definition} \label{def.def}Let $R=K\{x\}/I$ and $T=K\{t\}/J$ be two analytic $K$-algebras, $x=(x_1,...,x_n)$, $t=(t_1,...,t_q)$ with $I,J$ ideals and $K$ a real valued field. 
\begin{enumerate}
   \item  
   A {\em deformation of $R$ over $T$} is a flat morphism $\phi: T \to \kr$ of analytic $K$-algebras together with an isomorphism 
$\kr \otimes_T T/\langle t \rangle \to R$. The algebras $\kr$ (resp. $T$, resp. $\kr/\langle t \rangle \kr \cong R$) are called the (algebras of the) {\em total space} (resp. {\em base space}, resp. {\em special fiber}) of the deformation. We denote a deformation of $R$ by 
$(\phi,\iota): T\to \kr \to R$, with $\iota :\kr \to R$ the canonical projection, or just by $\phi: T\to \kr$. 
   \item 
   Let $(\phi',\iota'): T'\to \kr' \to R$ be a second deformation of $R$ over the analytic $K$-algebra $T'$. A {\em morphism} from $(\phi',\iota')$ to $(\phi,\iota)$ 
is given by a pair of morphisms $(\psi: \kr \to \kr', \varphi : T\to T')$
such that $\iota = \iota' \circ \psi$ and $\psi \circ \varphi = \phi' \circ \varphi$. Two deformations over the {\em same} base space $T$ are {\em isomorphic}  if there exists a morphism  $(\psi,\varphi)$ with $\psi$ an isomorphism and $\varphi = \id_T$.  
\item 
Let $(\phi,\iota): T\to \kr \to R$ be a deformation of $R$ over $T$, 
$\varphi : T\to S$ a morphism of analytic algebras, and let the lower square of the commutative diagram
$$
\begin{xymatrix}{
&R&\\
 \kr \hat\otimes_T S \ar[ur]^{\varphi^*\iota}  && \kr \ar[ll]_{\tilde\varphi} \ar[ul]_\iota \\
S \ar[u]^{\varphi^*\phi} && T \ar[ll]^\varphi \ar[u]_\phi  }
\end{xymatrix}
$$
be the analytic pushout\,\footnote{\,The analytic pushout can be constructed as the analytic tensor product $\kr \hat\otimes S$ (see \cite[Kapitel III, § 5]{GR71})
modulo the ideal generated by $\{(\phi(t)r,s)-(r,\varphi(t)s), t\in T, r \in \kr, s \in S\}$. It is the algebraic counterpart to the geometric pullback or fiber product.} of $\phi$ and $\varphi$. 

Then the morphism $\varphi^*\phi$ is flat (see e.g. \cite[Proposition 1.1.87]{GLS25}) and
there is a natural map 
$\varphi^*\iota:  \kr \hat\otimes_T S \to R$  such that 
$$(\varphi^*\phi, \varphi^*\iota) : S \to \kr \hat\otimes_T S \to R$$
is a deformation of $R$ over $S$. 
$(\varphi^*\phi, \varphi^*\iota)$ is called the 
{\em deformation induced from $(\phi, \iota)$ by $\varphi$}, and 
$\varphi$ is called the {\em base change map}.
$(\tilde \varphi,  \varphi)$ is a 
morphism of deformations from $(\varphi^*\phi, \varphi^*\iota)$ to
$(\phi, \iota)$. 
  \item    
  A deformation $(\phi,\iota): T\to \kr \to R$ of $R$ over $T$ is called {\em complete} if any deformation $(\psi,j): S\to \kq\to R$ of $R$ over some analytic algebra $S$ can be induced from $(\phi,\iota)$ by some base change map $T \to S$   (up to isomorphism of deformations of $R$ over $S$).
\item
 The deformation $(\phi,\iota)$ of $R$ over $T$ is called {\em versal} if it is complete and if the following lifting property holds:\\
Let  $(\psi,j)$ be a given deformation of $R$ over $S$. Let $k: S\to S'$ a surjection and $\varphi': T\to S'$ a morphism of analytic algebras, such that 
the induced deformations  $(\varphi'^* \phi, \varphi'^* \iota)$ and 
$(k^*\psi,k^*j)$ over $S'$ are isomorphic. Then there exists a morphism $\varphi: T \to S$ such that $k\circ \varphi = \varphi'$ 
and $(\psi,j) \cong (\varphi^* \phi, \varphi^* \iota)$.
\item 
A versal deformation is called {\em semiuniversal} or {\em miniversal} if, with the notations from 5.,  the cotangent map of
$\varphi$, $\dot\varphi: \fm_T/\fm_T^2 \to \fm_S/\fm_S^2$, 
 {$\fm_T, \fm_S$ the maximal ideals of $T, S$,}
is uniquely determined by $(\phi,\iota)$ and $(\psi,j)$. 
\end{enumerate}
\end{Definition}
Versality of $(\phi,\iota)$ means that $(\phi,\iota)$ is not only complete but in addition that any deformation $(\psi,j)$ can be induced from $(\phi,\iota)$ by a base change that  extends a given base change  inducing  $(k^*\psi,k^*j)$ from 
 $(\phi,\iota)$. $(\phi,\iota)$ is called {\it formally versal} if the lifting property 5. holds for Artinian analytic algebras $S$ and $S'$.
 
 Property 5. is needed to construct a versal deformation (completeness is not sufficient): starting from the trivial deformation $T=K \to \kr =R$
 one extends this to bigger and bigger Artinian base spaces  {(in a non-unique way)}. This can be done since the so called 'Schlessinger conditions'  (a kind of formal version of the lifting property 5.) hold for the deformation functor. In the limit one gets finally to a formal object, which is formally versal. \\
 
 The definition of versality applies also to unfoldings.
The classical notion of unfolding in the form we need, is as follows:

\begin{Definition} \label {def.unf}
 Let $f \in \fm \subseteq K\{x\}$. 
 \begin{enumerate}
 \item A {\em (p-parameter) unfolding} of $f$ is a power series 
 $F \in K\{x,s\}$, $s= (s_1,...,s_p)$, such that $F(x, 0) =f$. 
If $F(x,s) = f(x)$ we say that $F$ is the {\em constant unfolding} of $f$. $K\{s\}$ is called (the algebra of) the {\em parameter space} of $F$.
  \item Let $F \in K\{x,s\}$ and $G \in K\{x,t\}$, $t = (t_1,...,t_q)$,
 be two unfoldings of $f$. 
 A {\em right-left morphism from $G$ to $F$} is given by a tripel
 $(\Phi, \phi,\lambda),$ with
 \begin{enumerate}
  \item [(i)] $\Phi : K\{x,t\} \to K\{x,t\}$ a morphism, satisfying\\
  $\Phi(x_i) =: \Phi_i \in K\{x,t\}$, $\Phi_i(x,0) = x_i$, $i=1,...,n$,\\
  $\Phi(t_j) = t_j, \ j=1,...,q$,
 \item [(ii)] $\phi : K\{s\} \to K\{t\}$, a morphism, $\phi(s_j) =: \phi_j(t) \in K\{t \}, \ j=1,...,p$,   and   
\item [(iii)] $\lambda \in  K\{y,t\}$, $y=(y_1)$, $\lambda(y,0) = y$,   such that 
 $$\Phi(G)(x,t)= G(\Phi(x,t),t)  = \lambda(F(x,\phi(t)),$$
with $\Phi(x,t):=(\Phi_1(x,t),...,\Phi_n(x,t))$ and $\phi(t):=(\phi_1(t),...,\phi_p(t))$.
   \end{enumerate}
 If this holds, we say that  $G$ is {\em right-left induced}  from $F$.  If $\lambda$ is a translation, that is $\lambda(y,t) = y + \alpha(t),  \alpha \in K\{t\}, \alpha(0)=0$, then $G$ is called {\em right induced}  from $F$.
\item  An unfolding $F$ of $f$ is called {\em  right-left complete}  (resp. {\em right complete})\,\footnote{In the literature also "versal"
is used instead of "complete". We prefer to use versal only in the sense of Definition \ref{def.def}.}  if any unfolding of $f$ is induced from $F$ by a suitable right-left (resp. right) morphism. \\
$F$ is called {\em right-left versal} (resp. {\em right versal}) if in addition the lifting property from Definition \ref{def.def} (5) holds for right-left (resp. right) induced unfoldings.
A versal unfolding is called {\em semiuniversal} or {\em miniversal} if the cotangent map of the base change map is uniquely determined (see Definition \ref{def.def} (5)). 
 \end{enumerate}
 \end{Definition} 

\begin{Remark}\label{rm.unf}{\em
Since the Jacobian matrix of $\Phi$ at 0 is the identity in Definition \ref{def.unf}, $\Phi$ is in fact an automorphism of $K\{x,t\}$. Hence, if $G$ is right-left induced from $F$, then $G(x,t)$ is right-left equivalent to 
$H(x,t)$, where $H(x,t) := F(x,\phi(t))$ is obtained from $F$ by the base change $\phi$.\\
Usually the notion right equivalence is used if the translation is trivial, that is, $\alpha=0$, but the case  of a non-trivial translation is used in the proof of Theorem \ref{thm.split}.
The  translation $\alpha$  is  introduced to take care of the constant terms of $ F(x,{\phi(t)})$ for varying $t$. }
 \end{Remark}

\begin{Theorem}\label{thm.suunf}
Let $f\in \fm^2 \subseteq K\{x\}$ and assume that $\dim_K\fm/\langle\frac{\partial f}{\partial x_1},\ldots,\frac{\partial f}{\partial x_n}\rangle<\infty$.
Let $g_1,\ldots,g_p\in \fm$ be representatives of a generating system (resp. a basis) of $\fm/\langle\frac{\partial f}{\partial x_1},\ldots,\frac{\partial f}{\partial x_n}\rangle$ 
and $s=(s_1,\ldots,s_p)$ new variables.
Then 
$$F(x,s) := f(x)+\sum_{i=1}^pg_i(x)s_i\in K\{x,s\}$$ is a right-versal (resp. right-semiuniversal) unfolding  of $f$.
\end{Theorem}
 \begin{proof}
We only prove completeness, following \cite{GP25a} (using Grauert's Approximation Theorem \ref{grauertapprox}); versality can be proven in a similar way, but is more complicated in terms of notation.

 Let $G(x,t)\in K\{x,t\} $ be an arbitrary unfolding of $f$, $t=(t_1,\ldots,t_q)$.
 We have to prove that there exists $\phi(t)\in \langle t\rangle K\{t\}$, an automorphism  $\Phi$ of $K\{x,t\}$, $\Phi = (\Phi_1,...,\Phi_n)$, $\Phi_i = \Phi(x_i)$,
with $\Phi_i(x,0) = x_i$,
$\Phi(t_j) = t_j$, and $\alpha \in K\{t\}$ with  $\alpha(0)=0$,
such that
\begin{equation}\label{def1}
\Phi(G)=G(\Phi(x,t),t) = F(x,\phi(t)) +\alpha(t).
\end{equation}
A solution of order $e$ (in $t$) of (\ref{def1}) is a triple $(\Phi,\phi,\alpha) \in K\{x\}[t]^n \times K[t]^p  \times K[t]$  such that 
(\ref{def1}) holds mod $\langle t \rangle^{e+1}= \langle t \rangle^{e+1}K\{x,t\}$.

Since $\Phi(x,0) =x$, $\phi(0)= 0$, $\alpha(0)=0$  and $G(x,0)=f(x)=F(x,0)$ we obviously have a solution of (\ref{def1}) of order $0$. 
To apply Theorem \ref{grauertapprox}, we have to show that every solution 
$(\Phi,\phi,\alpha)$ of $(\ref{def1})$ of order $e$ in $t$ can be extended to a solution
$(\Phi',\phi',\alpha')$ of $(\ref{def1})$ of order $e+1$. 

Having the solution $(\Phi,\phi,\alpha)$ of order $e$, the difference
$G(\Phi(x,t),t) - F(x,\phi(t)) - \alpha(t)$ is a power series of order $\ge e+1$ in $t$. The homogeneous part of degree $e+1$ in $t$ of
$G(\Phi(x,t),t) - F(x,\phi(t)) -\alpha(t)$ can thus be written as
\[
 G(\Phi(x,t),t) - F(x,\phi(t)) -\alpha(t) = \sum_{| \nu | = e+1}t^{\nu}h_{\nu}(x) \mod \langle t\rangle^{e+2}
\]
with $h_{\nu} \in K\{x\}$. 
Define $g_0:=1$. By assumption $g_0,...,g_p$ generate $K\{x\}/\langle\frac{\partial f}{\partial x_1},\ldots,\frac{\partial f}{\partial x_n}\rangle$ and therefore we can write for all $\nu$ 
\[
h_{\nu} =  \sum_{j=0}^{p}b_{j \nu}g_j + \sum_{j=1}^nh_{\nu j}
\frac{\partial f}{\partial x_j}
\]
with $b_{j \nu } \in K$, $j \ge 1$,  and $h_{\nu j} \in K\{ x\}$.
We define 
\begin{align*}
\Phi'_j &= \Phi_j - \sum_{|\nu | = e+1}{}{h_{\nu j}}t^\nu \; \quad
j=1,\ldots,n\\
\phi'_j &= \phi_j + \sum_{|\nu | = e+1}b_{j \nu}t^\nu\;\quad j = 1, \ldots, p,\\
\alpha' &= \alpha +\sum_{|\nu | = e+1}b_{0 \nu}t^\nu.
\end{align*}
As $G(x,0)= f$ and $\Phi(x,0)=x$ we obtain 
$$\frac{\partial{(G(\Phi(x,t),t))}}{\partial{x_j}}\equiv \frac{\partial f}{\partial x_j} \mod \langle t\rangle .$$
This implies using Taylor's formula (see Remark \ref{rm.taylor})
\begin{align*}
G(\Phi'(x,t),t)&=  G(\Phi(x,t),t)-\sum_{j=1}^n \frac{\partial G}{\partial x_j}(\Phi(x,t),t)(\sum_{|\nu | = e+1}{}{h_{\nu j}}t^\nu) \mod \langle t\rangle^{e+2}\\
&=  G(\Phi(x,t),t)-\sum_{|\nu | = e+1}(\sum_{j=1}^nh_{\nu j}\frac{\partial f}{\partial x_j} )t^\nu \mod \langle t\rangle^{e+2}.
\end{align*}
On the other hand, by Taylor and the definition of $F$,
\begin{align*}
F(x,\phi'(t))& = F(x,\phi(t)) +\sum_{j=1}^p(\frac{\partial F}{\partial s_j}(x,\phi(t))(\sum_{|\nu | = e+1}b_{j \nu}t^\nu) \mod \langle t\rangle^{e+2}\\
&=F(x,\phi(t))+\sum_{|\nu | = e+1}(\sum_{j=1}^p b_{j\nu}g_j)t^\nu \mod \langle t\rangle^{e+2}.
\end{align*}
We obtain
$$G(\Phi'(x),t)-F(x,\phi'(t))-\alpha'(t)\equiv 0 \mod \langle t\rangle^{e+2}.$$
Now we can apply Theorem \ref{grauertapprox} to obtain an analytic solution of (\ref{def1}).
This proves the theorem.
\end{proof}

\begin{Remark}\label{rm.unf}{\em
We can of course define unfoldings $G$ of $f$ over arbitrary parameter spaces $K\{t\}/J$, $J\subseteq K\{t\}$ an ideal, as  an element
$G(x,t)\in (K\{t\}/J)\{x\} $ with $G(x,0)=f$.\\
 Then the power series $F$ from Theorem \ref{thm.suunf} is also right-versal for such unfoldings $G$:
Namely, let $G(t,x)\in (K\{t\}/J)\{x\} $ be an unfolding of $f$, $t=(t_1,\ldots,t_q)$. 
 We have to prove that there exists $\phi(t)\in \langle t\rangle$,
an automorphism  $\Phi$ of $(K\{t\}/J)\{x\}$  with $\Phi(x)_{t=0} = x$ and
$\Phi(t) = t$ and
an $\alpha \in K\{t\}$ with  $\alpha(0)=0$, satisfying
\begin{equation}\label{def2}
 \Phi(G)= G(\Phi(x,t),t)  = F(x,\phi(t)) +\alpha(t).
\end{equation}
To see this, we take a representative of $G(t,x)$ in $K\{t,x\}$ which  is an unfolding of $f$ over $K\{t\}$. We apply  Theorem (\ref{thm.suunf}) to this representative and pass then to  $(K\{t\}/J)\{x\} $.
}
\end{Remark}

We formulate now the general splitting lemma for power series (not necessarily with an isolated singularity) in $K\{x\}$ from \cite{GP25a}. Note that  the rank of the Hessian matrix $H(f):=\Big( \frac{\partial^2 f}{\partial x_i\partial x_j}(0)\Big)_{i,j=1,...,n}$ is invariant under right equivalence if $f \in \fm^2$.

\begin{Theorem}[Splitting lemma in any characteristic]
\label{thm.split}
Let $K$ be a real valued field and $f \in \fm^2 \subseteq K\{x\}$.
\begin{enumerate}
\item Let $\Char(K) \ne2$ and   $\rank H(f)=k$. Then 
 $$ f\ \rsim \ a_1x_1^2+\ldots+a_kx_k^2+g(x_{k+1},\dots,x_n) $$
   with $a_i \in K$,  $a_i \ne 0$, and
$g\in\langle x_{k+1},...,x_n\rangle^3$.
   $g$ is called the {\em residual
    part\/}\index{residual part} of $f$, it is uniquely determined up to   right equivalence in $K\{x_{k+1},...,x_n\}$.   
\item Let $\Char(K)=2$ and $\rank H(f)=2l$\ \footnote{In characteristic 2 the rank of the Hessian matrix is even.}. Then $f$ is right equivalent to 
$$ \sum_{i \text{ odd, } i=1}^{2l-1}(a_i x_i^2 + x_ix_{i+1} +
a_{i+1} x_{i+1}^2) + \sum_{i=2l+1}^{n}d_i x_i^2 + h(x_{2l+1},...,x_n),
$$
with  $a_i, d_i \in K$, $h\in \langle x_{2l+1},\ldots,x_{n} \rangle^3$. 
$g:=  \sum_{i=2l+1}^{n}d_i x_i^2  + h(x_{2l+1},\ldots,x_{n})$ is called the {\em residual part} of $f$, it is uniquely determined up to right equivalence in $K\{x_{2l+1},\ldots,x_{n}\}$.
\end{enumerate} 
\end{Theorem}

We formulate the uniqueness statement for the residual part explicitly and prove it separately (see also \cite[Theorem 2.1 and Theorem 3.5]{GP25}).
\begin{Proposition}\label{prop.split}
Let $f_0, f_1\in\fm^2\subseteq K\{x\}$ and assume that 
\begin{center}
$f_0 = q(x_1,...,x_{k})+g_0(x_{k+1},\dots,x_n) 
\rsim f_1 = q(x_1,...,x_{k})+g_1(x_{k+1},\dots,x_n),$
\end{center}
where
\begin{enumerate}
\item $\Char (K) \ne 2$:\\
$ q=a_1x_1^2+\ldots+a_kx_k^2, \ a_i \in K, a_i \ne 0$ and
$g_j(x_{k+1},\dots,x_n) \in \fm^3, \ j=0,1.$
  
\item$\Char(K) = 2$, \ $k=2l$:\\
$ q=\sum_{i \text{ odd, } i=1}^{k-1}(a_i x_i^2 + x_ix_{i+1} +
a_{i+1} x_{i+1}^2), \ a_i \in K,$ and
$g_j = \sum_{i=k+1}^{n}d_i x_i^2 + h_j, \ h_j\in \langle x_{k+1},\ldots,x_{n} \rangle^3, \ d_i \ne0, \ j=0,1.$
\end{enumerate}
Then $g_0 \rsim g_1$ in $K\{ x_{k+1},...,x_n\}$.
\end{Proposition}

If the field $K$ is algebraically closed, the coefficients $a_i$ and $d_i$ can be made to $1$. More precisely, we have

\begin{Corollary} \label{cor.split}
Let $f$ be as in Theorem \ref{thm.split}.
\begin{enumerate}
\item Let $char(K)\ne 2$ and assume that $K$ coincides with its subfield $K^2$ of squares. Then, with $k=\rank H(f),$ 
  $$ f\ \rsim x_1^2+\ldots+x_k^2+g(x_{k+1},\dots,x_n). $$
$g$ is uniquely determined up to  right equivalence in $K\{x_{2l+1},\ldots,x_{n}\}$. 
\item  Let $char(K) = 2$ and assume that quadratic equations are solvable in $K$. Then  $f$ is right equivalent to one of the following normal forms, with $g$ unique up to right equivalence in $K\{x_{2l+1},\ldots,x_{n}\}$, $2l=\rank H(f)$:
$$
\begin{array}{llll}
(a) & x_1x_2+x_3x_4+\ldots+x_{2l-1}x_{2l}+x_{2l+1}^2 &+g(x_{2l+1},\ldots,x_{n}), & 1\le  2l+1 \le n,\\
(b)& x_1x_2+x_3x_4+\ldots+x_{2l-1}x_{2l} &+ g(x_{2l+1},\ldots,x_{n}),  & 2 \le 2l \le n.
\end{array}
$$
\end{enumerate}
\end{Corollary}
\begin{proof}[Proof of Proposition \ref{prop.split}] 
1.  Let 
 $\varphi$  be an automorphism  of $K\{x\}$ such that
$\varphi (f_0)= f_1$. Then
$\varphi$  is given by  
$$\varphi(x_i) =: \varphi_i(x) =: l_i(x)+k_i(x), \ i=1,...,n,$$ 
with $k_i \in \fm^2$ and $l_i$ linear forms with $\det \big(\frac{\partial l_i}{\partial x_j}\big) \ne 0$. Then $\varphi (f_0)= f_1$ means
$$
  a_1\varphi_1^2+\ldots+a_k\varphi_{k}^2+ \ g_0(\varphi_{k+1},\ldots,\varphi_n) =
 a_1x_1^2+\ldots+a_kx_{k}^2 + \ g_1(x_{k+1},\ldots,x_n).
$$
Comparing the terms of  order 2 and of order $\ge 3$,  we get
\begin{align}
\tag{*}\label{*}
\begin{split}
& a_1l_1^2+\ldots+a_kl_{k}^2   =  a_1x_1^2+\ldots+a_kx_k^2 \ \text{ and}  \\
 &\sum_{ i=1}^{k}  a_ik_{i} (2l_{i} + k_{i}) + g_0(\varphi_{k+1},\ldots,\varphi_n) = g_1(x_{k+1},\ldots,x_n).
 \end{split}
 \end{align}
We set $F_i := 2l_{i} + k_{i}$, $i = 1,...,k$, and assume that there are $\psi_1,...,\psi_{k} \in K\{ x_{k+1},...,x_n \}$ satisfying 
$$ F_i (\psi_1,...,\psi_{k}, x_{k+1},...,x_n)=0, \ i = 1,...,k.$$
Then we  define the endomorphism $\varphi'$ of $K\{x_{k+1},...,x_n \}$
by
 $$\varphi'(x_i):= \varphi_i'(x_{k+1},\ldots,x_n) := 
 \varphi_i(\psi_1,...,\psi_{k},x_{k+1},\ldots,x_n), \ i = k+1,...,n.$$ 
This kills the terms $ a_ik_{i} (2l_{i} + k_{i})$
and by (\ref{*}) we get as required
$$g_0(\varphi'_{k+1},\ldots,\varphi'_n) =
g_1(x_{k+1},\ldots,x_n).$$

We have still to show that the endomorphism 
 $\psi$ exists and that $\varphi'$ is an automorphism of  $K\{ x_{k+1},...,x_n \}$.

\noindent We note that $(F_1,...,F_{k}, \varphi_{k+1},...,\varphi_{n})$ 
 is an automorphism of $K\{ x_1,...,x_n \}$ 
 since 
 $(\varphi_1,...,\varphi_n)$ is an automorphism and since $l_i$ is the linear part of $\varphi_i$.
 It follows that 
 $$\langle F_1, ...,F_k,  \varphi_{k+1},...,\varphi_{n} \rangle = \langle x_{1} ,...,x_{n}\rangle$$
  and if we  replace $x_i$ by $\psi_i(x_{k+1},...,x_n)$ for $i=1,...,k$, we get 
  $\langle 0, ..., 0, \varphi'_{k+1},...,\varphi'_{n} \rangle = \langle x_{k+1} ,...,x_{n}\rangle$. This shows that $\varphi'$ is an automorphism of $K\{ x_{k+1},...,x_n\}$.
  
  To show the existence of $\psi$, 
  we want to apply the  implicit function theorem (in the form of \cite[Theorem 1.1]{GP25}) to $F_1,...,F_{k}$. For this we must show $\det \big ( \frac{\partial F_i}{\partial x_j} (0)\big )_{i,j = 1,...,k} 
  =\det \big ( \frac{\partial l_i}{\partial x_j}\big )_{i,j = 1,...,k} \ne 0$. 
The quadratic terms of (\ref{*}) read
$$\ell:= a_1 l_1^2+\ldots+a_kl_{k}^2   =  a_1x_1^2+\ldots+a_kx_k^2,$$
 and  we get $\frac{\partial \ell}{\partial x_i} = 2a_ix_{i}$  for $i \le k$. 
  Since $\frac{\partial \ell}{\partial x_i} \in \langle l_{1} ,...,l_{k}\rangle$,
  it follows that $\langle x_{1} ,...,x_{k}\rangle \subseteq \langle l_{1} ,...,l_{k}\rangle$ and hence, with $l'_i(x_1,...,x_{k}) 
  := l_i (x_{1} ,...,x_{k},0,...,0)$, 
  we get $\langle x_{1} ,...,x_{k}\rangle = \langle l'_{1} ,...,l'_{k}\rangle$. Hence $\det \big ( \frac{\partial l_i}{\partial x_j}\big )_{i,j = 1,...,k} = \det \big ( \frac{\partial l'_i}{\partial x_j}\big )_{i,j = 1,...,k} \ne 0$.\\
  
2.  To see the uniqueness if $\Char(K)=2$
let $\varphi$  be a coordinate change given again by  
$\varphi(x_i) = \varphi_i(x) = l_i(x)+k_i(x), \ i=1,...,n,$
with $k_i \in \fm^2$ and $l_i$ linear forms with $\det \big(\frac{\partial l_i}{\partial x_j}\big) \ne 0$, such that
$\varphi (f_0)= f_1$. This means
\begin{align}
\tag{*}\label{*}
\begin{split}
 &\sum_{i \text{ odd, } i=1}^{2l-1}(a_i \varphi_i^2 + \varphi_i\varphi_{i+1} +
a_{i+1} \varphi_{i+1}^2) +  \sum_{i=2l+1}^{n}d_i \varphi_i^2 + \ h_0(\varphi_{2l+1},\ldots,\varphi_n) =\\ 
& \sum_{i \text{ odd, } i=1}^{2l-1}(a_i x_i^2 + x_ix_{i+1} +
a_{i+1} x_{i+1}^2) + \sum_{i=2l+1}^{n}d_i x_i^2 + \ h_1(x_{2l+1},\ldots,x_n).
\end{split}
\end{align}
By the uniqueness statement  for quadratic forms in characteristic 2 (see \cite[Theorem 3.1] {GP25}) we may assume,
after a linear coordinate change among the variables  $x_{2l+1},...,x_n$, that $l_{2l+1} = x_{2l+1},...,l_{n} = x_{n}$. Hence
\begin{center}
$\sum_{i=2l+1}^{n}d_i \varphi_i^2 = \sum_{i=2l+1}^{n}(d_i x_i^2 + d_ik_i^2).$
\end{center}
\noindent 
Comparing the terms of  order $\ge 3$ and setting 
$$K_i := k_i(a_ik_i +l_{i+1}+k_{i+1}) + k_{i+1}(l_i+a_{i+1}k_{i+1})$$
 we get
 $$\sum_{i \text{ odd, } i=1}^{2l-1} K_i + \sum_{i=2l+1}^{n}d_i k_i^2 \ +  \ h_0(\varphi_{2l+1},\ldots,\varphi_n)
 =  h_1(x_{2l+1},\ldots,x_n).$$
We define for $i=1,...,2l$, 
$$
\begin{array}{llll} 
&  F_i &:= l_i +a_{i+1}k_{i+1}, &\text{ if } i  \text{ is odd}\\
&  F_i &:= l_i + k_i  + a_{i-1}k_{i-1} &\text{ if } i \text{ is even}.\\
\end{array}
$$ 
Assume now that there are $\psi_1,...,\psi_{2l} \in K\{ x_{2l+1},\ldots,x_n\}$ satisfying 
$$F_i (\psi_1,...,\psi_{2l}, x_{2l+1},...,x_n)=0, \ i = 1,...,2l \text{ and } \psi(0)=0.$$
Replace $(x_1,...,x_{2l})$ by $(\psi_1,...,\psi_{2l})$ 
and define the endomorphism $\varphi'$ of $K\{ x_{2l+1},...,x_n\}$
by
 $$\varphi'(x_i):= \varphi_i'(x_{2l+1},\ldots,x_n) := 
 \varphi_i(\psi_1,...,\psi_{2l},x_{2l+1},\ldots,x_n), \ i = 2l+1,...,n.$$ 
This kills  
$\sum_{i \text{ odd, } i=1}^{2l-1} K_i 
$
and we get for the terms of order $\ge 3$
$$\sum_{i=2l+1}^{n}d_i k_i^2+h_0(\varphi'_{2l+1},\ldots,\varphi'_n) =
h_1(x_{2l+1},\ldots,x_n)$$
 and thus $g_0(\varphi'_{2l+1},\ldots,\varphi'_n) =
g_1(x_{2l+1},\ldots,x_n)$ by adding $\sum_{i=2l+1}^{n}d_i x_i^2$ on both sides.
\medskip

We have still to show that the endomorphism 
 $\psi$ exists and that $\varphi'$ is an automorphism of  $K\{x_{2l+1},...,x_n\}$.\\
The map $(l_1,\varphi_2, l_3,\varphi_4, ..., l_{2l-1},\varphi_{2l}, 
 \varphi_{2l+1},...,\varphi_{n})$ is an automorphism of \mbox{$K\{x_1,...,x_n\}$} since $(\varphi_1,...,\varphi_n)$ is an automorphism\footnote{We use the following. Let $\lambda=(\lambda_1,\ldots,\lambda_n):K\{x\}\longrightarrow K\{x\}$ be an endomorphism. The following conditions are equivalent:
 \begin{enumerate}
 \item $\lambda$ is an automorphism,
 \item $ \det \big ( \frac{\partial \lambda_i(0)}{\partial x_j}\big )_{i,j = 1,...,n} \ne 0$,
  \item  $\langle\lambda_1,\ldots,\lambda_n\rangle=\langle x_1,\ldots,x_n\rangle$.
  \end{enumerate}}
  and since $l_i$ is the linear part of $\varphi_i$.
 It follows 
 $$\langle l_1,\varphi_2, l_3,\varphi_4, ..., l_{2l-1},\varphi_{2l},  \varphi_{2l+1},...,\varphi_{n} \rangle = \langle x_{1} ,...,x_{n}\rangle$$
  and if we  replace $x_i$ by $\psi_i(x_{2l+1},...,x_n)$ for $i=1,...,2l$, we get 
  $\langle \varphi'_{2l+1},...,\varphi'_{n} \rangle = \langle x_{2l+1} ,...,x_{n}\rangle$ since $l_i(\psi)\in\langle x_{2l+1},\ldots,x_n\rangle^2$. This shows that $\varphi'$ is an automorphism of $K\{x_{2l+1},...,x_n\}$.
  
  To show the existence of $\psi$  we  apply the  implicit function theorem  to $F_1,...,F_{2l}$
  considered as elements of $K\{x_{2l+1},\ldots,x_n\}\{x_1,\ldots,x_{2l}\}$. To do this, we must show that \mbox{$\det \big ( \frac{\partial F_i}{\partial x_j} (0,0)\big )_{i,j = 1,...,2l}
  \ne 0$.} Since $k_i\in \fm^2$ we have 
  $$\frac{\partial F_i}{\partial x_j}(0,0)=\frac{\partial l_i}{\partial x_j}(0,0)=\frac{\partial \varphi_i}{\partial x_j}(0,0).$$
Comparing the quadratic terms of (\ref{*}), we get
 $$\ell:=\sum_{i \text{ odd, } i=1}^{2l-1}(a_i l_i^2 + l_ili_{i+1} +
a_{i+1} l_{i+1}^2 ) =
 \sum_{i \text{ odd, } i=1}^{2l-1}(a_i x_i^2 + x_ix_{i+1} +
a_{i+1} x_{i+1}^2),$$
 and $\frac{\partial \ell}{\partial x_i} = x_{i+1}$ if $i$ is odd and 
  $\frac{\partial \ell}{\partial x_i} = x_{i-1}$ if $i$ is even for $i \le 2l$. 
  Since $\frac{\partial \ell}{\partial x_i} \in \langle l_{1} ,...,l_{2l}\rangle$,
  it follows that $\langle x_{1} ,...,x_{2l}\rangle \subseteq \langle l_{1} ,...,l_{2l}\rangle$ and hence, with $l'_i(x_1,...,x_{2l}) 
  := l_i (x_{1} ,...,x_{2l},0,...,0)$, 
  we get $\langle x_{1} ,...,x_{2l}\rangle = \langle l'_{1} ,...,l'_{2l}\rangle$. Hence $$\det \big ( \frac{\partial F_i(0,0)}{\partial x_j}\big )_{i,j = 1,...,2l}=\det \big ( \frac{\partial l_i(0,0)}{\partial x_j}\big )_{i,j = 1,...,2l} = \det \big ( \frac{\partial l'_i(0,0)}{\partial x_j}\big )_{i,j = 1,...,2l} \ne 0.$$
  Now we apply the  implicit function theorem and obtain 
  $\psi_1,...,\psi_{2l} \in K\langle x_{2l+1},\ldots,x_n\rangle$ satisfying 
$F_i (\psi_1,...,\psi_{2l}, x_{2l+1},...,x_n)=0, \ i = 1,...,2l \text{ and } \psi(0)=0,$ as rquired.
\end{proof}

\begin{proof}[Proof of Theorem \ref{thm.split}] 
The uniqueness of the residual  part for statement 1. and 2. was proved in Proposition \ref{prop.split}. We prove now the existence of a splitting.\\
1. The existence of a splitting for char$(K) \ne2$ was proved in \cite[Theorem 2.1]{GP25} for  $f\in \C\{x\}$ or  $f\in \R\{x\}$  without assuming that $f$ has an isolated singularity (loc. cit. Theorem 2.4). The proof uses as an intermediate step the existence of a convergent semiuniversal unfolding (with non-trivial translation) for the quadratic part, similar as in item 2. for characteristic 2.
Since we proved the existence of a convergent semiuniversal unfolding in Theorem \ref{thm.suunf} for arbitrary real valued $K$, the proof given in \cite[Theorem 2.4]{GP25} works as well in our case and thus proves the splitting lemma in characteristic $\ne 2$.

2.  We prove now the existence of a splitting for convergent power series over real valued fields of characteristic 2. \\
By \cite[Theorem 3.1]{GP25}  (applied to the 2-jet of $f$) we can assume that
  $ f(x)= q + g$, where
$$ q(x_1,...,x_{2l}) = \sum_{i \text{ odd, } i=1}^{2l-1}(a_i x_i^2 +  x_ix_{i+1}+ a_{i+1} x_{i+1}^2)$$
and $g(x_1,...,x_n)=\sum_{i=2l+1}^{n}d_i x_i^2 +h(x_1,...,x_n)$,
with $a_i, d_i \in K$ and $h\in \fm^3 \subseteq K\{x\}$.\\
Setting $g'(x_1,...,x_{2l}):= g(x_1,...,x_{2l},0,...,0)$ we get
 $$g(x_{1},\dots,x_n)=g'(x_1,...,x_{2l}) + \sum_{i=2l+1}^{n}d_i x_i^2 +h'(x_1,...,x_n),$$
  with 
 $h'= \sum_{i=2l+1}^n x_ih_i(x_1,...,x_n)$. Then 
 $$f = q+g = 
 q+ g'+ \sum_{i=2l+1}^{n}d_i x_i^2 + h'=f'+  \sum_{i=2l+1}^{n}d_i x_i^2+ h',$$ 
 with
  $$f'(x_1,...,x_{2l}):=  f(x_1,...,x_{2l},0,...,0) =q(x_1,...,x_{2l}) + g'(x_1,...,x_{2l}).$$
 Since $j(q) = \langle x_1,...,x_{2l}\rangle$ and $d(q)=1$ we get from Theorem \ref{thm:finitedeterminacybound} that $q$ is 2-determined.
Since $g' \in \fm^3$, it follows that $f'$
 is right equivalent to $q$ by an automorphism $\varphi$ of 
 $K\{ x_1,...,x_{2l}\}$. Setting $\varphi_i(x_1,...,x_{2l}) := \varphi(x_i)$,
we have
 $$\varphi(f') = q(\varphi_1,...,\varphi_{2l})+g'(\varphi_1,...,\varphi_{2l}) = q(x_1,...,x_{2l}).$$
 Now define the automorphism  $\psi$ of $K\{ x_1,...,x_n\}$ by $x_i\mapsto \varphi_i(x_1,...,x_{2l})$ for
  $i=1,\dots,{2l}$, and \mbox{$x_i\mapsto x_i$} for $i>{2l}$. Then 
  $\psi(f') = \varphi(f') =q$ and
 we have
  $$\psi(f) = 
  q+\sum_{i=2l+1}^{n}d_i x_i^2 +\sum_{i={2l}+1}^n x_ih_i(\varphi_1,...,\varphi_{2l},x_{{2l}+1},...,x_n). $$
 Thus, after applying $\psi$, we may assume that 
 $$f= q +\sum_{i=2l+1}^{n}d_i x_i^2+ h', \ 
 h'= \sum_{i={2l}+1}^n x_ih_i(x_1,...,x_n).$$
  That is $f(x_1,...,x_{2l},0,...,0) = q$, in other words, $f$ is an unfolding of $q$.\\
  Since $\langle \frac{\partial q}{\partial x_1},..., \frac{\partial q}{\partial x_{2l}}\rangle =\langle x_1,...,x_{2l}\rangle$, $q$ is a semiuniversal unfolding of itself by Theorem \ref{thm.suunf}, and thus
the unfolding $f$ can be right induced from $q$.
That is,  there exists an automorphism  $\Phi =(\Phi_1,...,\Phi_n)$ of $K\{x\}$  with $\Phi_i=\Phi(x_i)$,
 $\Phi_i(x_1,...,x_{2l},0) = x_i$, $i=1,...,2l$,
  $\Phi(x_j) = x_j$, $j=2l+1,...,n$  
  and $\alpha \in K\{t\}$ with  $\alpha(0)=0$,
such that 
$$\Phi(f)=f(\Phi_1(x),...,\Phi_{2l}(x),x_{2l+1},...,x_n) = 
q(x_1,...,x_{2l}) +\alpha(x_{2l+1},...,x_n).$$ 
Now $\Phi(f)= \Phi(q)+\Phi(\sum_{i=2l+1}^{n}d_i x_i^2)+\Phi(h')
=q  +\sum_{i=2l+1}^{n}d_i x_i^2 + h$, where $h:=\Phi(h') \subseteq K\{x_{2l+1},...,x_n\}$. This proves the existence of a splitting  in statement 2.
\end{proof}


\section {Contact Simple Singularities in Positive Characteristic}\label{sec.class1}

This section is devoted to the classification of hypersurface singularities "without moduli" with respect to {\em contact equivalence}.

For the whole section let 
$K$ be an {\em algebraically closed} field  of characteristic $p \geq 0$. Points of an algebraic variety will be always {\em closed points} (if not said otherwise).
We present the classification of simple hypersurface singularities  $f \in K\{x\}$, 
\mbox{$x = (x_1,...,x_n)$},  w.r.t. contact equivalence, following \cite{GrKr90}, 
and  in the next section \ref{sec.class2} for right equivalence  following \cite{GN16}.\\

For $k\in \mathbb{N}$ let  
$$J^{(k)}:= K\{x\}/\fm^{k+1},$$
$\fm$  the maximal ideal of  K\{x\},  be space of $k$-jets of power series and 
$$\jet_k: K\{x\} \to J^{(k)}, f\mapsto f^{(k)}:= \jet_k(f)$$
 the canonical projection. 
We call
 $f^{(k)}$ the {\it k-jet of $f$}
 and we identify it
 with the power series expansion of $f$ up to (including) order $k$. 
 
 Let $G$ denote the right group $\kr= \Aut (K\{x\})$, respectively of the contact group $\kk=K\{x\}^* \rtimes \Aut (K\{x\})$. The k-jet of $G$ is
	$$G^{(k)}:=\{g^{(k)}=\jet_k(g)\mid g\in G\}.$$
For $g=\phi \in \kr$ let
 $\jet_k(g)= \phi^{(k)}$, 
$\phi^{(k)}(x_i) 
=\jet_k(\phi(x_i))$,  and
for $g=(u,\phi) \in \kk$  we have $\jet_k(g)= (u^{(k)},\phi^{(k)})$.
Then $G^{(k)}$ is an affine algebraic group with group structure given by $g^{(k)}\cdot  h^{(k)} = (gh)^{(k)}$, acting algebraically on the affine space $J^{(k)}$ via
\begin{align*}
	&G^{(k)}\times J^{(k)}\to J^{(k)}, \hskip4pt
	\left(g^{(k)}, f^{(k)}\right)\mapsto (g\cdot f)^{(k)} = (u \phi(f))^{(k)},
\end{align*}
if $g = (u,\phi)$, i.e. we let representatives act and then take the $k$-jets. \\

\begin{Definition}[Simple singularity]\label{def.simple}
We say that $f\in \fm$ is {\em  right--simple} (resp. {\em contact--simple}) if there 
 exists a neighbourhood $U$ of $f$ in $K\{x\}$ such that $U$ intersects only finitely many orbits $Gf_1,...,Gf_s$ where $G$ is the right group $\kr= \Aut (K\{x\})$, respectively the contact group $\kk=K\{x\}^* \rtimes \Aut (K\{x\})$.  
 \end{Definition}
 
By the definition of the topology on $K\{x\}$ ($V\subseteq K\{x\}$ is open iff  $\jet_k(V)$ is open in $ J^{(k)}$ for all $k$ (equivalently, for  $k\ge$ some $k_0$)) this means that for each $k$  there exists  a Zariski-open neighbourhood $U_k$ of the $k$-jet
$f^{(k)}$ in $J^{(k)}$ such that $U_k$ decomposes into the finitely many $G^{(k)}$--orbits $G^{(k)}f_1^{(k)},...,G^{(k)}f_s^{(k)}$ (with $s$ independent of $k$).\\

For the classification we have to consider unfoldings of $f\in \fm$ parametrized by (the coordinate ring of) an affine space and not  just by a power series as in Definition \ref{def.unf}.
We consider  $F(x,t) : = \sum a_\alpha(t) x^\alpha \in K[{t}]\{x\}$,
$t=(t_1,...,t_k)$, as a family of power series such that  $F_t(x) = F(x,t) \in K\{x\}$ for fixed $t\in K^k$. 
We say that $F$ is an {\em unfolding (with trivial section)} of $f$ over $K^k$ or over $K[t]$, if $F(x,0) =f$ and if $F_t \in \fm$ for $t\in K^k$. Then $F$ can be written as
$$F_{t}(x)=F(x,{t})=f(x)+\sum\limits^k_{i=1} t_ih_i(x,{t}) \in K[{t}]\{x\},$$
with $h_i(x,{t}) \in \fm$ for fixed $t$.

We say that $f\in \fm$ is of {\em finite (countable) deformation type}\index{finite!deformation type}
for right (resp. contact) equivalence, if
there exists a finite (countable) set $\{g_j\}\subseteq K\{x\}$ of power series such that for every unfolding $F(x,t)$
of $f$ over $K^k$ as above
there exists a Zariski neighbouhood $U=U(0)\subseteq K^k$ such that for each ${t} \in U$ there exists a $g_j$ with $F_{t}\overset{r }{\sim} g_j$ (resp. $F_{t}\ \overset{c }{\sim} \ g_j$).

\begin{Lemma}\label{lem:simple-is-isolated}
Let $f\in \fm$ be of finite deformation type for right (resp. contact) equivalence. Then $f$ (resp. $K\{x\}/\langle f\rangle$) is an isolated singularity. The same conclusion holds if $f$ is right (resp. contact) simple.
\end{Lemma}
\begin{proof}
Let $f$ be a non-isolated singularity, i.e., $\mu(f) =\infty$ (resp. $\tau(f) = \infty$).  Then 
$\fm^k$ is not contained in $j(f)$ (resp. in $tj(f) =j(f) + \langle f \rangle$) for all $k$ and hence there exists a $k_0$ and an $i$, w.l.o.g. $i=1$, such that $x_1^l \notin j(f)$ (resp. $x_1^l \notin tj(f)$ ) for $l\geq k_0$. 

We claim that $f$ can be
deformed into (isolated) singularities of arbitrary high Milnor (resp. Tjurina) number. Namely, for $p=0$ or $p \nmid N $ the Milnor (and Tjurina) number of  $g_N= x_1^N+... + x_n^N$ is finite. We consider the family
$$f_{t,N} = t_0 \cdot g_N+ t_1\cdot f, \ t=(t_0:t_1) \in \P^1,$$
which is,  for generic $t\in \P^1$, an unfolding of $g_N$ as well as of $f$ with constant Milnor number $\mu_N = \mu(f_{t,N})= \mu(1/t_0 \cdot f_{t,N}) =\mu(g_N + t_1/t_0 \cdot f) \le \mu(g_N) < \infty$. This follows from the semicontinuity of $\mu$ (Proposition \ref{prop.semcont}).  We have $x_1^l  \notin j(f)=j(t_1/t_0 \cdot f)$. 
Now $f$ deforms in $f_{t,N}$, which has an isolated singularity with Milnor number  $\mu_N \ge N $, for infinitely many $N$ and hence infinitely many $\mu_N$. This shows that $f$ is not of finite right deformation type.
An analogous argument with $\tau_N$ shows that $f$ is not of finite contact deformation type.

$f$ is also not right (resp. contact) simple. In fact, for a given $k$ there are infinitely many $N>k$ 
such that the Milnor (resp. Tjurina) numbers of $f_{t,N}$, with $t$ generic and fixed, are different. That is, each neighbourhood of $f$ intersects infinitely many
right (resp. contact) orbits.
\end{proof}

\begin{Lemma}\label{lem:findeftype} The following  are equivalent for $f\in K\{x\}$:
\begin{enumerate}
\item $f$ right  (resp. contact) simple.
\item $f$ is of finite deformation type for right (resp. contact) equivalence. 
\end{enumerate}
\end{Lemma}

\begin{proof}
The proof follows from the previous Lemma \ref{lem:simple-is-isolated} and with the help of the semiuniversal deformation of $f$ (Theorem \ref{thm.sudef}).
For details we refer to \cite[Lemma 3.3.31]{GLS25}.
\end{proof}

Before we present the classification of simple singularities, we mention the explicit form of the semiuniversal deformation (Definition \ref{def.def}), which is needed in the proof of Lemma \ref{lem:findeftype} and for the classification.

\begin{Theorem}[Semiuniversal deformation]\label{thm.sudef} Let $K$ be a real valued field (not necessarily algebraically closed) and
$f\in K\{x\}$ with $\tau(f)<\infty$. Let $g_1,\dots,g_{\tau}\in K\{x\}$ be representatives of
a $K$-basis (respectively a system of generators) of the
 Tjurina algebra
$$ T^1_R:= K\{x\}/\textstyle\big\langle f,\frac{\partial
f}{\partial x_1},\dots,\frac{\partial f}{\partial
x_n}\big\rangle$$
of $R= K\{x\}/\langle f \rangle$.
Setting 
$$F(x,t)= f(x)+\sum_{j=1}^{\tau} t_jg_j(x) \in K\{x,t\}$$
and
$\kr =  K\{x,t\} / \langle F \rangle$, then 
$K\{t\} \to \kr \to R$
is a semiuniversal (respectively versal) deformation of
$R$ over $K\{t\}$.
\end{Theorem}

For a proof see \cite[Corollary 6.5]{GP25}. It needs Grauert's approximation theorem for real valued fields $K$ of arbitrary characteristic. This is proved in \cite{GP25}, as well as the existence of a semiuniversal deformation for arbitrary analytic $K$-algebras (not only hypersurfaces) with isolated singularity.\\
 
The most important classification result for hypersurface singularities is the following result by V. Arnold for $\C\{\bx\}$, \cite{AGV}, which can be extended to algebraically closed fields of characteristic 0:

\begin{Theorem}[Simple singularities in characteristic 0] \label{thm:ADE-char0}
Let $K$ be an algebraically closed, real valued  field of characteristic 0 and $f\in \fm \cdot K\{x_1,...,x_n\}$.
\begin {enumerate}
\item  $f$ is contact--simple $\Leftrightarrow f$ is contact equivalent to an ADE singularity from the following list (with $q=x_3^2+\cdots + x_n^2$):
\medskip

$
\begin{array}{llll}
A_k: & x_1^{k+1} + x_2^2    &+\ q,\ \ \ \ \ k\geq 1\\
D_k: & x_2(x_1^2+x_2^{k-2}) &+\ q, \ \ \ \ \ k\geq 4\\
E_6: & x_1^3+x_2^4   &+ \ q\\
E_7: & x_1(x_1^2+x_2^3) &+\ q\\
E_8: & x_1^3+x_2^5  &+\ q
\end{array}
$
\medskip

\item $f$ is right--simple $\Leftrightarrow f$ is contact--simple. 
\end {enumerate}
\end {Theorem}
Thus, if $\Char(K)=0$ or in the complex analytic case we do not distinguish between right simple and contact simple and just say simple.

\begin{proof}
Statement 1. follows from the classification of Theorem \ref{thm:GrKr} for $\Char(K) \ge 0$.  For the proof of statement 2. we only need to show that the ADE singularities from 1. are right simple, since it is clear that right simple implies contact simple.
 We note that the ADE singularities are all quasihomogeneous \index{quasihomogeneous} and hence right equivalence coincides with contact equivalence (see \cite[Lemma1.2.13\
 ]{GLS25}, which works for $K$ algebraically closed of  characteristic 0), finishing the proof. 
\end{proof}

The classification of hypersurface singularities in positive characteristic started with the  paper \cite[Theorem 1.4]{GrKr90}.

\begin{Theorem}[Contact-simple singularities in characteristic $\ge$ 0] \label{thm:GrKr}  Let $K$ be a real valued, algebraically closed field of characteristic $\ge 0$. 
The following are equivalent for a hypersurface singularity 
\mbox{$f \in  K\{x_1, \ldots, x_n\}$}.
\begin{enumerate}
\item [(1)] $f$ is contact--simple,
\item [(2)] $f$ is an ADE--singularity (with few extra normal forms in small characteristics). That is, $f$ is contact equivalent to a power series form the following lists I.1 -- II.3.
\end{enumerate}
\end{Theorem}

\begin{itemize}
\item[I.] char$(K) \neq 2$\\
I.1 Dimension 1 \\

\centerline{\begin{tabular}[20pt]{|c|l l|}
\hline 
\ Name \ & \multicolumn{2}{l|}{\ Normal form for $f \in K\{x,y\}$}\\
\hline
        & & \\[-2ex]
$\mathrm{A}_k$&$\quad \quad\ \ \ x^2+y^{k+1}$&$  k\ge 1$\\
        && \\[-2ex]
\hline
        && \\[-2ex]
$\mathrm{D}_k$&$\quad\quad\ \ \ x^2y+y^{k-1}$ & $ k\ge 4$\\
        && \\[-2ex]
\hline
        && \\[-2ex]
$\mathrm{E}_6$& $\ E_6^0$ \quad $x^3+y^4$&\\
        && \\[-2ex]
      & $\ E_6^1$ \quad $x^3+y^4+x^2y^2$ & additionally in char = 3 \ \\
      && \\[-2ex]
\hline
        && \\[-2ex]
$\mathrm{E}_7$& $\ E_7^0$ \quad $x^3+xy^3$&\\
        && \\[-2ex]
  & $\ E_7^1$ \quad $x^3+xy^3 +x^2y^2$& additionally in char = 3 \ \\
        && \\[-2ex]
 \hline
        && \\[-2ex]       
$\mathrm{E}_8$&\ $E_8^0$ \quad $x^3+y^5$&\\
       && \\[-2.5ex]        
&
$\begin{rcases}
      \ E_8^1$  \quad $x^3+y^5+x^2y^3 & \\
      \ E_8^2$  \quad $x^3+y^5+x^2y^2 
\end{rcases}$
  & additionally in char = 3 \\
   & $\ E_8^1$  \quad $x^3+y^5+x^2y^2$& additionally in char = 5 \\ [3pt]       
\hline
\multicolumn{3}{c}{ }\\
\end{tabular}}

I.2 Dimension $\ge 2$\\

\centerline{\begin{tabular}[20pt]{|cl|}
\hline 
\multicolumn{2}{|c|}{\ Normal form for $f \in K\{x_1,...,x_n\}$ \ } \\
\hline 
        &  \\[-2ex]
$\ g(x_1,x_2)+x_3^2+\ldots+x_n^2$& \\[3pt]
\hline
\multicolumn{2}{c}{ }
\end{tabular}}
where $g \in K\{x_1,x_2\}$ is one of the singularities in Table I.1. The name of $f$ is that of $g$. \\

\item[II.]  char$(K) = 2$\\
II.1 Dimension 1\\

\centerline{\begin{tabular}[20pt]{|l|l l|}
\hline 
\ Name \ & \multicolumn{2}{l|}{\ Normal form for $f \in K\{x,y\}$}\\
\hline
        & & \\[-2ex]
$\ \mathrm{A}_{2m-1}$& $\quad\quad\quad\quad \ x^2+xy^m$ & \
    $m\ge 1$\\
        && \\[-2ex]
 $\ \mathrm{A}_{2m}$& $\ A_{2m}^0 \quad\quad \ x^2+y^{2m+1}$ &  \   
   $m\ge 1$\\
        && \\[-2ex]
  & $\ A_{2m}^r  \quad\quad \ x^2+y^{2m+1}+xy^{2m-r}$ & \
     $m\ge 1$,  $1\le r \le  m-1$ \ \\
    && \\[-2ex]
\hline
        && \\[-2ex]
$\ \mathrm{D}_{2m}$& $\quad\quad\quad\quad \ x^2y+y^{m}$ & \
$ m\ge 2$\\
        && \\[-2ex]
$\ \mathrm{D}_{2m+1}$& $\ D_{2m+1}^0  \quad \ x^2y+y^{2m}$ & \
$ m\ge 2$\\
        && \\[-2ex]
 &$\ D_{2m+1}^r  \quad \ x^2y+y^{2m}+xy^{2m-r}$ & \
    $ m\ge 2$,  $1\le r \le  m-1$ \ \\
        && \\[-2ex]
\hline
        && \\[-2ex]
$\ \mathrm{E}_6$& $\ E_6^0$ \quad\quad \ \ $x^3+y^4$&\\
        && \\[-2ex]
      & $\ E_6^1$ \quad\quad \ \ $x^3+y^4+xy^3$&  \\
      && \\[-2ex]
\hline
        && \\[-2ex]
$\ \mathrm{E}_7$& $\quad\quad\quad\quad x^3+xy^3$&\\
        && \\[-2ex]
 \hline
        && \\[-2ex]       
$\ \mathrm{E}_8$& $\quad\quad\quad\quad x^3+y^5$&\\
\hline
\multicolumn{3}{c}{ }\\
\end{tabular}}

II.2 Dimension 2

\centerline{\begin{tabular}[20pt]{|l|l l|}
\hline 
\ Name \ & \multicolumn{2}{l|}{\ Normal form for $f \in K\{x,y,z\}$}\\
\hline
        & & \\[-2ex]
$\ \mathrm{A}_{k}$& $\quad\quad\quad\quad \ z^{k+1}+xy$ & \
    $k\ge 1$\\
     && \\[-2ex]
\hline
        && \\[-2ex]
$\ \mathrm{D}_{2m}$& $\ D_{2m}^0  \quad\quad  z^2+x^2y+xy^{m}$ & \
$ m\ge 2$\\
        && \\[-2ex]
 & $\ D_{2m}^r  \quad\quad  z^2+x^2y+xy^{m}+xy^{m-r}z$ & \
    $ m\ge 2$,  $1\le r \le  m-1$ \ \\
        && \\[-2ex]
$\ \mathrm{D}_{2m+1}$& $\ D_{2m+1}^0  \quad \ z^2+x^2y+y^{m}z$ & \
$ m\ge 2$\\
        && \\[-2ex]
 &$\ D_{2m+1}^r  \quad \ z^2+x^2y+y^{m}z+xy^{m-r}z$ & \
    $ m\ge 2$,  $1\le r \le  m-1$ \ \\
        && \\[-2ex]
\hline
        && \\[-2ex]
$\ \mathrm{E}_6$& $\ E_6^0$ \quad\quad \ \ $z^2+x^3+y^2z$&\\
        && \\[-2ex]
      & $\ E_6^1$ \quad\quad \ \ $z^2+x^3+y^2z+xyz$&  \\
      && \\[-2ex]
\hline
        && \\[-2ex]
$\ \mathrm{E}_7$& $\ E_7^0$ \quad\quad \ \ $z^2+x^3+xy^3$&\\
        && \\[-2ex]
& $\ E_7^1$ \quad\quad \ \ $z^2+x^3+xy^3+x^2yz$&  \\
      && \\[-2ex]
& $\ E_7^2$ \quad\quad \ \ $z^2+x^3+xy^3+x^3z$&  \\
      && \\[-2ex]
& $\ E_7^3$ \quad\quad \ \ $z^2+x^3+xy^3+xyz$&  \\
      && \\[-2ex]
 \hline
        && \\[-2ex]       
$\ \mathrm{E}_8$& $\ E_8^0$ \quad\quad \ \ $z^2+x^3+y^5$&\\
        && \\[-2ex]
& $\ E_8^1$ \quad\quad \ \ $z^2+x^3+y^5+xy^3z$&  \\
      && \\[-2ex]
& $\ E_8^2$ \quad\quad \ \ $z^2+x^3+y^5+xy^2z$&  \\
      && \\[-2ex]
& $\ E_8^3$ \quad\quad \ \ $z^2+x^3+y^5+y^3z$&  \\
      && \\[-2ex]
& $\ E_8^4$ \quad\quad \ \ $z^2+x^3+y^5+xyz$&  
 \\[3pt]
\hline
\multicolumn{3}{c}{ }\\
\end{tabular}}

II.3 Dimension $\ge 3$\\

\centerline{\begin{tabular}[20pt]{|ll|}
\hline 
\multicolumn{2}{|l|}{\ Normal form for $f \in K\{x_1,...,x_n\}$ \ } \\
\hline 
        &  \\[-2ex]
$\ g(x_1,x_2)+x_3x_4+\ldots+x_{2k-1}x_{2k},$ &\ $n=2k, \ k\ge 2 \ $  
\\[3pt]
\hline
    &  \\[-2ex]
$\ g(x_1,x_2,x_3)+x_4x_5+\ldots+x_{2k}x_{2k+1},$ &\ $ n=2k, \ k\ge 2 \ $
 \\[3pt]
\hline
\multicolumn{2}{c}{ }\\
\end{tabular}}
where $g \in K\{x_1,x_2\}$ resp. $K\{x_1,x_2,x_3\}$ is one of the singularities in Table II.1. resp. II.2. The name of $f$ is that of $g$. 
\end{itemize}

For the proof of Theorem \ref{thm:GrKr} we refer to \cite{GrKr90}, where it is proved for formal power series in $K[[x]]$. The proof for $K\{x\}$ is basically the same.

\begin{Remark} {\em
\begin{enumerate}
\item The normal forms in dimension 2 are exactly the normal forms of rational double points which were classified by Artin \cite{Ar77}. Moreover, Lipman \cite{Lip1} showed that a two-dimensional double point is rational if and only if it is absolutely isolated, i.e. can be resolved by a finite sequence of blowing up points. This criterion is used for the proof that ADE-singularities are of finite deformation type.

\item The normal forms in dimension 1 are exactly the normal forms of power series $f$ which \\
\text{   } (a) are reduced, \\
\text{   } (b) have multiplicity 2 or 3 and \\
\text{   } (c) the reduced total transform of $f$ after one blowing up has also property (b). \\
This was proved by Kiyek and Steinke \cite{KS85}.
\item The upper index 0 denotes the classical normal form, which is the most special with respect to deformations.
\end {enumerate}
}
\end{Remark}

The proof of Theorem \ref{thm:GrKr} yields also a characterization of the non-isolated singularities $A_\infty$ \index{$Ainf$@$A_\infty$}
and $D_\infty$, \index{$Dinf$@$D_\infty$} which are natural limits of $A_k$ and $D_k$ if $k$ tends to infinity.

\begin{Definition}  \label{Ainf}
$f \in K[[x_1,...,x_n]]$ is of type 
${A}_\infty$ resp. ${D}_\infty$ if it is contact equivalent to one of the following normal forms:\\

char(K) $\ne 2$\\

\centerline{\begin{tabular}[20pt]{|l|l l|}
\hline 
\ Name \ & \multicolumn{2}{l|}{\ Normal form for $f \in K\{x_1,...,x_n\}$}\\
\hline
         && \\[-2ex]
$\ \mathrm{A}_\infty$& $\ \   x_2^2+ \ldots +x_n^2$  
&\\[3pt]
\hline
        && \\[-2ex]
$\ \mathrm{D}_\infty$& $\ \ x_1^2x_2+x_3^2+ \ldots + x_n^2$ 
&\\[3pt]
\hline
\multicolumn{2}{c}{ }\\
\end{tabular}}

char(K) $= 2$\\

\centerline{\begin{tabular}[20pt]{|l|l l|}
\hline 
\ Name \ & \multicolumn{2}{l|}{\ Normal form for $f \in K\{x_1,...,x_n\}$}\\
\hline
         && \\[-2ex]
$\ \mathrm{A}_\infty$& $\quad\quad\quad\quad  x_2^2+x_3x_4+ \ldots +x_{2k-1}x_{2k},$& \ $n=2k$  \\ 
&& \\[-2ex]
&$\quad\quad\quad\quad x_2x_3+x_4x_5+ \ldots +x_{2k}x_{2k+1},$& \ $n=2k+1$  \\ 
&& \\[-2ex]
\hline
        && \\[-2ex]
$\ \mathrm{D}_\infty$& $\ \ D_{\infty}^0  \quad\quad  x_1^2x_2+x_3x_4+ \ldots + x_{2k-1}x_{2k},$ & \  $n=2k$ \\
        && \\[-2ex]
                             &   $\quad\quad\quad\quad x_1^2x_2+x_3^2+ x_4x_5\ldots + x_{2k}x_{2k+1},$ & \  $n=2k+1$\\
        && \\[-2ex]
                             &   $\ \ D_{\infty}^m  \quad\quad x_1^2x_2+x_3^2+ x_1x_2^mx_3+x_4x_5\ldots + x_{2k}x_{2k+1},$ & \  $n=2k+1, m \ge 1$                           
\\[3pt]
\hline
\multicolumn{3}{c}{ }\\
\end{tabular}}
\end{Definition}

For the characterization of  $A_\infty$
and $D_\infty$ we introduce the following definition:
 Call $f\in K\{x\}$ {\it 0-modular}\index{0-modular}  or of {\it 0-modular deformation type}, if for each $k$ there is an open neighbourhood $U_k$ of the $k$-jet $f^{(k)}$ in  the space of $k$-jets $J^{(k)}$ such that $U_k$ meets only finitely many contact orbits of $k$-jets (their number may depend on $k$). Note that 0-modular hypersurface singularities are necessarily of {\em countable deformation type}\index{countable!deformation type}\index{deformation type!countable}, the converse being true for not countable fields (cf.  \cite[Theorem B]{BGS87}) (and meaningless for countable fields).
 Isolated singularities are of {\em finite deformation type}\index{deformation type!finite} iff they are 0-modular (since they are finitely determined) but for non-isolated singularities the following theorem (due to \cite[Theorem 1.6]{GrKr90}) applies.

\begin{Theorem}[Characterization of $A_\infty$ and $D_\infty$] \label{AD-infty}
The following statements about $f\in K\{x_1,..., x_n\}$ are equivalent: 
\begin {enumerate}
\item $f$ is of type $A_\infty$ or $D_\infty$.
\item f is 0-modular but not of finite deformation type.
\end {enumerate}
\end{Theorem}

\begin{proof} It is not difficult to see that $A_\infty$ deforms  (only) to $A_k, k\ge 1$ and $D_\infty$ only to $A_k$ and $D_k$, $k\ge 4$. Hence  they are 0-modular but not of finite  deformation type. 

For the onverse we assume that $\Char(K) \ne 2$. Assume that $f$ is 0-modular, but neither simple nor $A_\infty$ or $D_\infty$.  Since the deformation theory of $f$ and $f+ x_{n+1}^2$ is essentially the same, we may assume that mult$(f)>2$. 
If $n \ge 3$ or if $n=2$ and mult$(f)>3$, the initial part of $f$ defines a projective variety which has moduli, hence the singularity defined by $f$ itself is not 0-modular. In case $n=2$, mult$(f)=3$, we may assume by \cite[Lemma 3.4]{BGS87} that, in suitable coordinates 
$x, y$
$$f(x, y)=x^3+ax^2y^2+bxy^4 +cy^6$$
with $a, b, c \in K\{x, y\}$. Replacing $x^3$ by 
$x(x+ \lambda y^2)(x+\lambda^2y^2)$, $\lambda \in K$, gives a family of non-isomorphic singularities (blowing up two times exhibits exactly four points on a rational component of the exceptional divisor, whose cross-ratio varies with $\lambda$). This shows that $f$ is not 0-modular, contradicting the assumption. Thus $f$ is either simple or $A_\infty$ or $D_\infty$.
The case $\Char(K) = 2$ can be proved using similar arguments as in \cite[Section 3.5]{GrKr90}.
\end{proof}


\section {Right Simple Singularities in Positive Characteristic}\label{sec.class2}

In this section let $K$ be be an algebraically closed, real valued field of characteristic $p>0$ (if not said otherwise). We classify right simple hypersurface singularities $f\in \mathfrak{m}\subseteq K\{x\}=K\{x_1,\ldots,x_n\}$, following mainly  \cite{GN16} and \cite{Ng12}. For the classification of contact simple singularities see the previous Section \ref{sec.class1}. In contrast to $\mathrm{char}(K)=0$, where the classification of right simple and contact simple singularities coincides (Theorem \ref{thm:ADE-char0}), the right classification is very different from the contact classification in positive characteristic. For example, for every fixed $p>0$, there are only finitely many classes of right simple singularities. For $p=2$ and $n$ even only the $A_1$-singularity 
$x_1x_2+x_3x_4+\ldots+x_{n-1}x_n$
is right simple, while for $n>1$ and odd no right simple singularity exists.

By Definition \ref{def.simple}  $f$ is right simple iff there is a neighbourhood $U$ of $f$ in $K\{x\}$ which meets only finitely many $\mathcal{R}$-orbits, with $\kr = \Aut_KK\{x\}$ the right group.  
By Lemma \ref{lem:findeftype} this means that $f$ is of finite deformation type for right equivalence.

The classification of right simple singularities is summarized in Theorems \ref{thm1.5.0}, \ref{thm1.5.1} and \ref{thm1.5.2}.
Note that $f\in\mathfrak{m}\setminus\mathfrak{m}^2$ iff the Milnor number (Definition \ref{def:mu_pos}) satisfies $\mu(f)=0$, and then $f\overset{r }{\sim} x_1$ (hence right simple) by the implicit function theorem. We may therefore assume in the following that $f\in\mathfrak{m}^2$. 

We end this section with a conjecture about the relation between the Milnor number $\mu$ and the right-modality if $\mu$ goes to infinity.\\
 
We start with the classification of univariate power series  
$$f = \sum_{n\ge0} c_n x^n \in K\{x\},$$
$K\{x\}$ being  the ring of convergent power series in one variable $x$. Since $f=ux^m$, $m$ the order of $f$ and $u$ a unit, and since any deformation of $f$ has an order  $\le m$, it follows that $f$ is contact simple.
This is however not the case for right equivalence.\\

The complete classification with respect to  right equivalence of univariate singularities in positive characteristic of {\em any $\kr$-modality} \footnote{The $G$-modality of $f \in K\{x\}$ is, roughly speaking, at most $k$, if a neighbourhood of $f$ in $K\{x\}$ can be covered by finitely many at most $k$-dimensional families of orbits of the group $G$. Hence $f$ is $G$-simple iff the $G$-modality is 0.
For a rigorous definition of modality we refer to \cite{GN16} or \cite[Section 3.3.2]{GLS25}.}, $\kr$ the right goup, was given in  \cite{Ng12} (see also  \cite{Ng13}). We present here only the main results.

Let  $\supp(f) = \{n\mid c_n\ne 0\}$, and $\mt(f)=\min \{n \mid n \in \supp{f}\}$ the {\em multiplicity} or {\em order} of $f$. For each $n\in \N$ we set 
\begin{enumerate}
\item $e(n):=\max \{i\mid p^i \text{ divides } n\}$
\item $e:=\min \{e(n)\mid n \in \supp(f)\}$
\item $q:= \min \{n \in \supp(f)\mid e(n)=e\}$
\item $k:=1$ if $\mt(f)=q$, otherwise \\
$k := \max\{k(n)\mid \mt(f) \le n<q, n\in
\supp(f) \}$, 
where $k(n) := \lceil \frac{q-n}{p^{e(n)}-p^e} \rceil $
\end{enumerate}

\begin{Proposition} \label{prop.5.1}
With the notatios above, if $\mu(f) < \infty$ then $d = q+p^e(k-1)$ is exactly the right determinacy of $f$.
\end{Proposition} 
For the proof of the proposition see \cite[Proposition 2.8]{Ng12}. 
 Note that $q = \mu(f) + 1$  is the first exponent in the expansion of $f$ which is not divisible by $p$. In particular, if $p\nmid \mt(f)$ then $\mt(f)$ is the right determinacy of $f$ and $f\overset{r }{\sim} x^{\mt(f)}$.

Furthermore,  in \cite[Theorem 2.11]{Ng12}  the author proves a precise normal form of univariate power series (which we omit here, since the formulation is quite complicated).

Another main result of that paper is the following theorem.

\begin{Theorem}\label{thm1.5.0}
Let $f\in\mathfrak{m}^2\subseteq K\{x\}$ be a univariate power series such that its Milnor number $\mu:=\mu(f)$ is finite. Then the right modality of $f$ is
$$\mathcal{R}\text{-}\mathrm{mod}(f)=[\mu/p], $$
the integer part of $\mu/p$. \\
In particular, $f$ is right simple if and only if $\mu<p$, and then $f\overset{r }{\sim} x^{\mu+1}$.
\end{Theorem}

\begin{proof}
For the full proof of the theorem we refer to \cite[Thm. 3.1]{Ng12}.
Here we prove only the second part, i.e. $f$ is right simple iff $\mu<p$ and then $f\overset{r }{\sim} x^{\mu+1}$. 

The ``if''-statement follows easily from the upper semi-continuity of the Milnor number and Proposition \ref{prop.5.1}: If $p\nmid \mt(f)$ (in particular, if $\mu<p$) then $f\overset{r }{\sim} x^{mt(f)}$.

It suffices to show that if $\mu\geq p$ then $\mathcal{R}\text{-}\mathrm{mod}(f)\geq 1$. Indeed, since $\mu\geq p$, $m:=\mt(f)\geq p$. 

If $m=p$ then we may assume that $f=x^p+a_{p+1}x^{p+1}+\ldots\in K\{x\}$. Consider the unfolding
$$f_t=f+tx^{p+1}=x^p+(t+a_{p+1})x^{p+1}+\ldots.$$
of $f$. We show that  $f_t\overset{r }{\sim} f_{t'}$ implies $t=t'$. If $\varphi(x)=u_1x+u_2x^2+\ldots$ is in $\mathcal{R}$ then $u_1\neq 0$ and 
$$\varphi(f_t)=u_1^px^p+u_2^px^{2p}+\ldots+ (t+a_{p+1}) u_1^{p+1}x^{p+1}+\ldots.$$
If $\varphi(f_t)=f_{t'}$ then $u_1^p=1$, hence $u_1=1$, and $t=t'$. Thus $f_t$
is a 1-dimensional family of orbits in a neighbourhood  of $f$, which  implies that $\mathcal{R}\text{-}\mathrm{mod}(f)$, the right modality of $f$, is $\geq 1$.

Now, assume that $m>p$ and consider the unfolding $g_t:=G(x,t):=f+t\cdot x^{p}$ of $f$ at $0$ over $\mathbb A^1$. By the semicontinuity of the modality (\cite[Proposition 3.3.14]{GLS25}), there exists an open neighbourhood $V$ of $0$ in $\mathbb A^1$ such that 
$\mathcal{R}\text{-}\mathrm{mod}(g_{t})\leq \mathcal{R}\text{-}\mathrm{mod}(f)$ for all $t\in V$. Take a $t_0\in V\setminus \{0\}$, then the above case with $mt=p$ yields that $\mathcal{R}\text{-}\mathrm{mod}(g_{t_0})\geq 1$ since $mt(g_{t_0})=p$, and hence
$\mathcal{R}\text{-}\mathrm{mod}(f)\geq \mathcal{R}\text{-}\mathrm{mod}(g_{t_0})\geq 1.$
\end{proof}

We summarize now the classification of right simple singularities $f\in \mathfrak{m}\subseteq K\{x\}=K\{x_1,\ldots,x_n\}$  with $n \ge 2$, following \cite{GN16}. 
The following theorem is due to \cite [Theorem 3.2]{GN16}.

\begin{Theorem}[Right simple singularities in characteristic $>$ 2]\label{thm1.5.1}
Let $\Char(K)=p>2$. 
\begin{itemize}
\item[(i)] A plane curve singularity $f\in \mathfrak{m}^2\subseteq K\{x,y\}$ is right simple if and only if it is right equivalent to one of the following normal forms\\

\centerline{\begin{tabular}[20pt]{|c|l l|}
\hline 
\ Name\ & \multicolumn{2}{l|}{\ Normal form}\\
\hline
         && \\[-2ex]
$\mathrm{A}_k$&$\ x^2+y^{k+1}$&$1\leq k\leq p-2$\\
\hline
         && \\[-2ex]
$\mathrm{D}_k$&$\ x^2y+y^{k-1}$& $4\leq k< p$\\
\hline
         && \\[-2ex]
$\mathrm{E}_6$& $\ x^3+y^4$&$3<p$\\
\hline
         && \\[-2ex]
$\mathrm{E}_7$&  $\ x^3+xy^3$&$3<p$\\
\hline
         && \\[-2ex]
$\mathrm{E}_8$& $\ x^3+y^5$&$5<p$\\
\hline
\end{tabular}}
\bigskip

\item[(ii)] A hypersurface singularity $f\in\mathfrak{m}^2\subseteq K\{x_1,\ldots,x_n\},n\geq 3,$ is right simple if and only if it is right equivalent to one of the following normal forms\\

\centerline{\begin{tabular}[20pt]{|cl|l}
\hline 
\multicolumn{2}{|l|}{\ Normal form} \\
\hline 
      & \\[-2ex]
$\  g(x_1,x_2)+x_3^2+\ldots+x_n^2$\ & \ where $g$ is one of the singularities in Table \ref{thm1.5.1} (i) \
 \\[3pt]
\hline
\multicolumn{2}{c}{ }\\
\end{tabular}}
\end{itemize}
\end{Theorem}

The following theorem and its corollaries are due to \cite [Theorem 3.3, Corollary 3.4 and 3.5]{GN16}.
\begin{Theorem}[Right simple singularities in characteristic 2]\label{thm1.5.2}
Let $p=\mathrm{char}(K)=2$. A hypersurface singularity $f\in\mathfrak{m}^2\subseteq K\{x_1,\ldots,x_n\}$ with $n\geq 2$, is right simple if and only if $n$ is even and if it is right equivalent to 
$$\mathrm{A}_1 : x_1x_2+x_3x_4+\ldots+x_{n-1}x_{n}.$$
\end{Theorem}

The following interesting corollary follows immediately from the classification of right simple singularities.

\begin{Corollary}\label{coro1.5.1}
For any $p>0$ there are only finitely many right simple singularities $f\in\mathfrak{m}^2\subseteq K[[x_1,\ldots,x_n]]$ (up to right equivalence). For $p=2$, either no or exactly one right simple singularity exists. 
\end{Corollary}

The corollary also shows that if $f_k,k\geq 1$, is any sequences of simple singularities in positive characteristic then the sequence of Milnor numbers $\mu(f_k),k\geq 1$, is bounded. Note that this is wrong in characteristic zero since the $A_k,D_k,k\geq 1$, with Milnor number $k$, are all simple. We like to pose the following conjecture:

\begin{conjecture}\label{conj:rightmodality}\index{Conjectures and problems}\index{Problems and conjectures}
Let $p>0$ and $f_k\in K\{x_1,\ldots,x_n\}$ be a sequence of isolated singularities with Milnor number going to infinity if $k\to \infty$. Then the right modality of $f_k$ goes to infinity. 
\end{conjecture}

\end{document}